\documentclass[12pt,reqno]{amsart}
\usepackage{amsmath,amssymb,amsthm,bbm,fullpage,tikz,xypic,mathdots}
\usepackage{braket,caption,subcaption,cancel}

\usetikzlibrary{arrows,decorations.markings}

\newtheorem{theorem}{Theorem}
\newtheorem{corollary}[theorem]{Corollary}

\newtheorem{lemma}[theorem]{Lemma}
\newtheorem{proposition}[theorem]{Proposition}
\newtheorem*{bfc}{Boson-Fermion Correspondence}

\theoremstyle{remark}

\newtheorem*{remark*}{Remark}

\theoremstyle{definition}

\newtheorem*{example*}{Example}

\newtheorem{definition}{Definition}

\renewcommand\appendix{\Alph{section}}

\begin{document}

\title{On Hamiltonians for six-vertex models}


\author{Ben Brubaker}
\address{Department of Mathematics, 127 Vincent Hall, 206 Church St. SE, Minneapolis, MN 55455}
\email{bbrubake@math.umn.edu}

\author[Andrew Schultz]{Andrew Schultz}
\address{Department of Mathematics, Wellesley College, 106 Central Street, Wellesley, MA 02482}
\email{andrew.c.schultz@gmail.com}

\begin{abstract}
In this paper, we explain a connection between a family of free-fermionic six-vertex models and a discrete time evolution operator on one-dimensional Fermionic Fock space. The family of ice models generalize those with domain wall boundary, and we focus on two sets of Boltzmann weights whose partition functions were previously shown to generalize a generating function identity of Tokuyama.  We produce associated Hamiltonians that recover these Boltzmann weights, and furthermore calculate the partition functions using commutation relations and elementary combinatorics. We give an expression for these partition functions as determinants, akin to the Jacobi-Trudi identity for Schur polynomials.
\end{abstract}

\date{\today}

\maketitle

\section{Overview}

Hamiltonians arising from Fock representations of Clifford algebras 
were explored by the Kyoto school, for example in \cite{Kyoto-III, Kyoto-Summary, JimboMiwa}. The Boson-Fermion correspondence
gives an explicit isomorphism between this Fermionic Fock representation and a polynomial algebra, the Bosonic Fock space. The image 
of elements under this correspondence are commonly called ``$\tau$-functions.'' The Kyoto school papers show that $\tau$-functions are solutions
to integrable hierarchies of nonlinear differential equations. Moreover, the Bosonic Fock space may be identified with the ring of symmetric
functions over a field. In particular if the Clifford algebra is $\mathfrak{gl}(\infty)$, then there exists a simple family of $\tau$-functions equal to Schur polynomials. 

Thinking of each application of the Hamiltonian operator as a step in discrete time, the evolution of a one-dimensional model gives rise to a
two-dimensional lattice model. In the case above of the Hamiltonian for $\mathfrak{gl}(\infty)$, the resulting two-dimensional model is the five-vertex
model; a nice exposition of this fact may be found in Zinn-Justin \cite{Zinn-Justin}. Our first result is a generalization of this fact, using a deformation
of the above Hamiltonian operator for $\mathfrak{gl}(\infty)$ whose evolution produces the six-vertex model studied in \cite{bbf-ice}. These models have
boundary conditions generalizing the more familiar domain wall boundary conditions.

As a consequence, the partition functions for these six-vertex models may be studied using methods for evaluating $\tau$-functions, e.g., the time evolution of fermionic fields (given in Proposition~\ref{commutation}) and Wick's theorem. This theme is also present in \cite{Zinn-Justin}, and we borrow many of the same techniques for analyzing our more complicated six-vertex model.
The partition function of our family of six-vertex models was computed (upon using a combinatorial bijection with Gelfand-Tsetlin patterns) by Tokuyama \cite{tokuyama} and subsequently reproved using the Yang-Baxter equation in \cite{bbf-ice}. Here we give an alternate proof of Tokuyama's result using only commutation relations for Hamiltonians and some elementary combinatorial facts. Moreover, we provide new explicit expressions for partition functions as a determinant, akin to the Jacobi-Trudi identity for ordinary Schur polynomials.

The modified $\tau$-functions we study involve a Hamiltonian arising from the theory of super Clifford algebras.  Super Clifford algebras and their (super) Boson-Fermion correspondence were studied by Kac and Van de Leur in \cite{KacVandeLeur}, and the Hamiltonian operators we present in Definition~\ref{superhamdef} are their $\Gamma_{\pm}(x,y)$ in Proposition 4.4 of \cite{KacVandeLeur}. They did not pursue explicit formulas for the resulting $\tau$-function, but combining Wick's theorem and Jacobi-Trudi type formulas identifies them as (skew) supersymmetric Schur polynomials, which we detail in an appendix at the end of the paper.

Variants of the six-vertex models in \cite{bbf-ice} for Cartan types $B$ and $C$ were explored in \cite{bbcg} and \cite{ivanov}, respectively. We will explain how Hamiltonian techniques are likewise applicable to these cases.

Our study of partition functions for lattice models using Hamiltonian operators was motivated by results in the representation theory of $p$-adic groups.
The Shintani-Casselman-Shalika formula is an explicit expression for the spherical Whittaker function of an unramified principal series on a (quasi-split) reductive group over a $p$-adic field. While we refrain from fully defining it here, it is a matrix coefficient on the space of the principal series of particular importance in automorphic forms. Most notably, for the group $\text{GL}(n)$, its special values are precisely the partition function of the six vertex model in \cite{bbf-ice, hamel-king}.
The approach to using Hamiltonians to study Whittaker functions over {\it real} reductive groups was used previously by Kazhdan and Kostant \cite{Kostant, Etingof}, who observed that the Laplacian on the group, restricted to the space of Whittaker functions, is the quantum Toda Hamiltonian. This resulted in a proof of the total integrability of the Toda lattice. In the theory of automorphic forms the role of the Laplacian at real places is replaced by Hecke operators at the $p$-adic places, and a approach to these via Hamiltonians has been put forward in papers of Gerasimov, Lebedev, and Oblezin, e.g., \cite{GLO-baxter}. Viewed in terms of two-dimensional lattice models, these operators act on the column parameters of the model, while ours act on the rows. Given these connections, we hope the results of this paper are a first step toward a Hamiltonian realization of a natural class of operators for representations of $p$-adic groups.

This work was supported by NSF grant DMS-1406238.

\section{Preliminaries and Statement of results}

\subsection{The Boson-Fermion correspondence and $\tau$-functions} To give a precise statement of our results, we first define the relevant Hamiltonian operators from the Clifford algebra $\mathfrak{gl}(\infty)$ and recall the Boson-Fermion correspondence in this case. Further information can be found in \cite{JimboMiwa}, \cite{AlexandrovZabrodin}, and \cite{Zinn-Justin}. Our notational conventions more closely follow this latter reference -- in particular we locate fermions at half-integers.

Let $A$ denote the Clifford algebra generated by creation and annilation operators $\psi_i^\ast$ and $\psi_i$ with $i \in \mathbb{Z}-\frac{1}{2}$ and
satisfying relations
\begin{equation} [ \psi_i, \psi_j ]_+ = 0, \quad [ \psi_i, \psi_j^\ast ]_+ = \delta_{i,j}, \quad [\psi_i^\ast, \psi_j^\ast]_+ = 0, \label{pluscomm} \end{equation}
where $[ x, y]_+ = xy + yx$.  We create generating functions on these operators: 
\begin{align*}
\psi(z) = \sum_{k \in \mathbb{Z}-\frac{1}{2}} \psi_{-k} z^{k-\frac{1}{2}} \quad &\quad 
\psi^*(z) = \sum_{k \in \mathbb{Z}-\frac{1}{2}} \psi_{k}^* z^{k+\frac{1}{2}}.
\end{align*}

Then let $\mathcal{W}$ be the set of ``free-fermions'' defined by
$$ \mathcal{W} := \oplus_{i \in \mathbb{Z}} \mathbb{C} \psi_i \oplus \oplus_{i \in \mathbb{Z}} \mathbb{C} \psi_i^\ast $$
with subspaces
$$ \mathcal{W}_{\text{ann}} :=  \oplus_{i < 0} \mathbb{C} \psi_i^\ast \oplus \oplus_{i > 0} \mathbb{C} \psi_i \quad 
\mathcal{W}_{\text{cr}} := \oplus_{i > 0} \mathbb{C} \psi_i^\ast \oplus \oplus_{i < 0} \mathbb{C} \psi_i. $$
Let $\mathcal{F}$ denote the left $A$-module $A / A \mathcal{W}_{\text{ann}}$ and let $\mathcal{F}^\ast$ denote the right $A$-module $A \mathcal{W}_{\text{cr}} \backslash A$; these are termed the ``Fock representations of $A$.'' They are cyclic $A$-modules whose generators are typically denoted with bra-ket notation $| 0 \rangle$ mod $A \mathcal{W}_{\text{ann}}$ and $\langle 0 |$ mod $A \mathcal{W}_{\text{cr}}$, respectively, for some choice of ``vacuum vector'' 0. There is a symmetric bilinear form on $\mathcal{F}^\ast \otimes_A \mathcal{F}$ denoted by $\langle 0 | a \otimes b | 0 \rangle$.

The space $\mathcal{F}$ is also a $\mathfrak{gl}(\infty)$ module, where we define
$$ \mathfrak{gl}(\infty) := \mathbb{C} \cdot 1 \oplus \left\{ \sum_{i,j} a_{i,j} : \psi_i^\ast \psi_j : \; \mid \; \exists \, N \, \text{such that} \, a_{i,j} = 0 \text{ if $| i - j| > N$} \right\}, $$
with $: \psi_i^\ast \psi_j :$ regarded as abstract symbols satisfying a Lie bracket defined as in (1.6) of \cite{JimboMiwa} and $1$ is a central element. The notation $: \psi_i^\ast \psi_j :$ is used elsewhere to mean the ``normal ordering'' with respect to the vacuum vector $| 0 \rangle$ defined by
$$ : \psi_i^\ast \psi_j : = - : \psi_j \psi_i^\ast : = \begin{cases} \psi_i^\ast \psi_j & \text{if $i > 0$} \\  - \psi_j \psi_i^\ast & \text{if $i < 0$} \end{cases} $$
which is useful in eliminating trivial infinite quantities from the definitions.
The action of
$X = \sum_{i,j} a_{i,j} : \psi_i^\ast \psi_j :$ on an element $a | 0 \rangle \in \mathcal{F}$ is given by
$$ X \cdot a | 0 \rangle := \text{ad} \, X (a) | 0 \rangle + a \cdot \sum_{i > 0 > j} a_{i,j} \psi_i^\ast \psi_j | 0 \rangle  $$

For each $n \in \mathbb{Z}$ define the elements
$$ J_n = \sum_{i \in \mathbb{Z}-\frac{1}{2}} : \psi_{i-n}^\ast \psi_{i} : $$
which satisfy the commutation relations $[J_m, J_n] = m \delta(m+n) \cdot 1$ where $\delta$ denotes the Kronecker delta symbol. 
This allows us to define, associated to infinitely many parameters $t = (t_1, t_2, \ldots)$ an action by the element
\begin{equation}\label{eq:hamiltonian.intro} H[t] := \sum_{n=1}^\infty t_n J_n \quad \text{and} \quad e^{H[t]} := \sum_{k=1}^\infty \frac{H[t]^k}{k!} \end{equation}
on $\mathcal{F}$. 

The element $J_0$ is central in $\mathfrak{gl}(\infty)$ and so we may decompose $\mathcal{F}$ as a direct sum of $J_0$-eigenspaces $\mathcal{F}_\ell$ indexed by integer eigenvalues $\ell$. These are irreducible representations of $\mathfrak{gl}(\infty)$ with highest weight vector $| \ell \rangle$ defined by
\begin{equation} | \ell \rangle = \begin{cases} \psi_{\ell+\frac{1}{2}} \cdots \psi_{-\frac{1}{2}} | 0 \rangle & \text{if $\ell < 0$} \\
1 \Ket{0} & \text{if $\ell = 0$} \\ \psi^\ast_{\ell-\frac{1}{2}} \cdots \psi^\ast_{\frac{1}{2}} \Ket{0} & \text{if $\ell > 0$}. \end{cases} \label{ellstate} \end{equation} 
There is an analogous story for the the Fock space $\mathcal{F}^\ast$ with highest weight representations indexed by $\Bra{\ell}$.

\begin{bfc} The following map is an isomorphism of vector spaces from $\mathcal{F}_\ell$ to $V_\ell$, where each $V_\ell \simeq \mathbb{C}[t]$:
$$ a \mid 0 \rangle \longmapsto \langle \ell \mid e^{H[t]} a \mid 0 \rangle. $$
\end{bfc}
Jimbo and Miwa \cite{JimboMiwa} justify this by noting that the image consists of Schur functions, which give a basis for the polynomial ring in infinitely many variables $t = (t_1, t_2, \ldots)$. In particular, given $\ell \in \mathbb{Z}$ and $\lambda$ a non-increasing sequence $\lambda_1 \geq \lambda_2 \geq \cdots \geq \lambda_n \geq 0$, we define 
\begin{equation} \Ket{\lambda;\ell} = \Ket{\psi^*_{\ell+\lambda_1-\frac{1}{2}} \psi^*_{\ell+\lambda_2 - \frac{3}{2}} \cdots \psi^*_{\ell+\lambda_n-n+\frac{1}{2}}|\ell-n}. \label{partitionwithshift} \end{equation}
(Note that if we pad the partition $\lambda$ with additional $0$ parts, the state of free fermions is unchanged.) 
Then the Schur function is realized as
$$ \Bra{\ell} e^{H[t]} \Ket{\lambda; \ell} =: s_\lambda[t], $$
where the variables $t_q$ in $t = (t_1, t_2, ...)$ are (up to a simple factor of $\frac{1}{q}$) equal to the power sum symmetric functions. Thus making the change of variables:
\begin{equation} t_q = \frac{1}{q} \sum_{i=1}^n x_i^q \label{simplemiwa} \end{equation}
gives the symmetric function $s_\lambda$ as a symmetric polynomial in the variables $x_1, \ldots, x_n$. Think of this change of variables as $t_q = \frac{1}{q} \text{tr}(X^q)$ for the matrix $X \in \mathfrak{gl}(n, \mathbb{C})$ with eigenvalues $(x_1, \ldots, x_n)$; this substitution appears in this form as Equation (1.3) in Miwa \cite{Miwa}.

As in Section~2 of \cite{JimboMiwa}, any polynomial in $t$ expressible in the form
$$ \tau_{\ell}(t; g) := \Bra{\ell} e^{H[t]} g \Ket{\ell} $$
for some $g = \text{exp}( \sum_{i,j} a_{i,j} \psi_i^\ast \psi_j)$ will be referred to as a $\tau$-function. Our main results of the next section will be explicit expressions for variants of the $\tau$-functions under generalizations
of the Miwa transformation in \eqref{simplemiwa}.

\subsection{$\tau$-functions for superalgebra Hamiltonians}

In order to describe our results on $\tau$-functions, first consider the following generalizations of the transformation in \eqref{simplemiwa}, and
the subsequent Hamiltonian operators that arise from them.

\begin{definition} \label{superhamdef}
For $\mathbf{x} = (x_1,\cdots,x_n)$ and $q \geq 1$, define
\begin{align*}
s_q^+ := \frac{1}{q} \sum_{i=1}^n x_i^q \quad \quad \quad \quad 
s_q^- := \frac{1}{q} \sum_{i=1}^n x_i^{-q}.
\end{align*}
Let $H_\pm[s] = \sum_{q \geq 1} s_q^\pm J_{\pm q}$.  Then define $\phi_\pm(x_i)$ by $e^{H_\pm[s]} = \prod_{i=1}^n e^{\phi_\pm(x_i)}.$

Similarly, to an infinite set of variables $t = (t_1, t_2, \ldots)$, let
\begin{align*}
s_q^+(t) := \frac{1}{q} \sum_{j=1}^n(1- (-t_j)^{q})  x_j^q \quad \quad \quad \quad 
s_q^-(t) := \frac{1}{q} \sum_{j=1}^n(1- (-t_j)^{q}) x_j^{-q}.
\end{align*}
Let $H_\pm[s(t)] = \sum_{q \geq 1}s_q^\pm(t)J_{\pm q}$, and define $\phi_\pm(x_i;t_i)$ by $e^{H_\pm[s(t)]} = \prod_{j=1}^n e^{\phi_\pm(x_j;t_j)}.$
\end{definition}

Just as $t_q$ in \eqref{simplemiwa} could be regarded as a trace, the elements $s_q^{\pm}(t)$ may be viewed as supertraces (as defined, for example, in Section I.1.5 of \cite{Kac}). Indeed if $V = V_{\bar{0}} \oplus V_{\bar{1}}$ is a
$\mathbb{Z} / 2 \mathbb{Z}$ graded vector space with basis given by a union of bases for $V_{\bar{0}}$ and $V_{\bar{1}}$, then an operator $a \in \text{End}(V)$, viewed as a Lie superalgebra, has block form
$$ a = \begin{pmatrix} \alpha & \beta \\ \gamma & \delta \end{pmatrix} \quad \text{with supertrace} \quad \text{str}(a) := \text{tr}(\alpha) - \text{tr}(\delta). $$ 
Thus, for example, $s_q^+(t) = \frac{1}{q} \text{str}(a^q)$ in $\mathfrak{gl}(n | n)$ where the eigenvalues of $a$ are $x_1, \ldots, x_n$ in the $V_{\bar{0}}$ component and $-t_1 x_1, \ldots, -t_n x_n$ in the $V_{\bar{1}}$ component. As noted in the Overview, these Hamiltonians arise in the study of superalgebras and are denoted $\Gamma_{\pm}(x,y)$ in Proposition 4.4 of \cite{KacVandeLeur}, where we have set $y = -tx$.

By combining Wick's theorem and the Jacobi-Trudi formula for (skew) supersymmetric Schur functions, we obtain the following formula for the $\tau$-function of the superalgebra Hamiltonian in Definition~\ref{superhamdef}:
\begin{equation} \langle \mu ; n-1 \mid e^{H_+[s(t)]} \mid \lambda; n \rangle =  \det_{1 \leq p,q \leq n} h_{\lambda_q-\mu_p-q+p}[x \mid y] = s_{\lambda / \mu}(x | tx), \label{plainsuperhamtau} \end{equation}
with $h_k$ as defined in \eqref{splithk}. The first equality above, using Wick's theorem, is no harder than for the $\mathfrak{gl}(\infty)$ Hamiltonian in \eqref{eq:hamiltonian.intro} as it is obtained from a formal change of variables. The second equality may be found on p.~22 of Macdonald \cite{Mac-themevar} or, in the non-skew case, Equation~(1.7) in Moens and Van der Jeugt~\cite{MoensVanderJeugt}. Details for the first equality above are presented in an appendix at the end of the paper.

In what follows, we make a small but very important change in the $\tau$-functions in \eqref{plainsuperhamtau}. We replace $e^{H_{+}[s(t)]} = \prod_{i=1}^n e^{\phi_+(x_i;t_i)}$ by $\prod_{i=1}^n \left[ e^{\phi_+(x_i;t_i)} \psi_{-1/2} \right]$ in the bra-ket, inserting the annihilation operator $\psi_{-1/2}$ between each application of $e^{\phi_+(x_i;t_i)}$; we make a similar adjustment to $e^{H_-[s(t)]}$. This radically changes the analysis for the resulting $\tau$-functions. Perhaps they should no longer be referred to as ``$\tau$-functions'' since these are typically expectation values of 
$e^{H_{\pm}[t]}$ multiplied by group-like elements -- exponentiated linear combinations of elements $\psi^\ast_i \psi_k$ (see for example Section~3 of \cite{AlexandrovZabrodin}). By contrast, the operators $e^{\phi_+(x_i;t_i)} \psi_{-1/2}$ destroy one particle with each application. Remarkably, this change is precisely what is needed for the bra-ket to match the partition function of two-dimensional ice-type lattice models studied in \cite{bbf-ice}.

To state our results, we introduce one further shorthand notation. In the event that $\lambda$ is strictly decreasing, we have that $\lambda-\rho = [\lambda_1-n,\cdots,\lambda_n-1]$ is non-decreasing, in which case we write $$\Ket{\lambda} := \Ket{\lambda-\rho;n}.$$  Observe that one can also view $\Ket{\lambda}$ as the coefficient of $\mathbf{z}^\lambda := \prod z_i^{\lambda_i}$ in $\psi^*(z_1)\cdots \psi^*(z_n)\Ket{0}$.

Our main results are the following.

\begin{theorem}\label{th:computing.weight.of.one.step}
 Let $\lambda = (\lambda_1, \ldots, \lambda_n)$ and $\mu = (\mu_1, \ldots, \mu_{n-1})$ be strictly decreasing, and let $k>\lambda_1$. Then 
\begin{align*}
\Braket{ \mu | e^{\phi_+(x;t)} \psi_{-1/2} | \lambda } &= \left\{\begin{array}{ll}(-1)^n t^{r(\lambda;\mu)} (1+t)^{s(\lambda; \mu)-r(\lambda;\mu)+1} x^{|\lambda|-|\mu|} & \mbox{ if }\lambda_i \geq \mu_i \geq \lambda_{i+1} \text{ for all }i, \\ [5pt]0 &\mbox{otherwise}.\end{array}\right. \\[10pt]
\Braket{ \mu | \psi_{k-\frac{1}{2}} e^{\phi_-(x;t)} | \lambda } &= \left\{\begin{array}{ll}t^{l(\lambda;\mu)}(1+t)^{s(\lambda; \mu)-r(\lambda;\mu)+1}x^{|\lambda|-|\mu|-k}& \mbox{ if }\lambda_i \geq \mu_i \geq \lambda_{i+1} \text{ for all }i, 
\\[5pt]0& \mbox{ otherwise}.\end{array}\right. 
\end{align*}
where $r(\lambda; \mu) = \left|\{1 \leq i \leq n-1: \mu_i = \lambda_{i+1}\}\right|$; $l(\lambda;\mu) = \left|\{1 \leq i \leq n-1: \mu_i = \lambda_i\}\right|$; and $s(\lambda; \mu) = \left|\{1 \leq i \leq n-1: \lambda_i > \mu_i\}\right|$. The notation $| \lambda |$ means the sum of the parts of the partition $\lambda$.
\end{theorem}

The previous result tells us how to compute the weight associated with the evolution of $\lambda$ to $\mu$ from one application of our (shifted) Hamiltonians.  If $\lambda$ has $n$ parts, it follows that the only state which can appear after $n$ successive applications of, for example, the $e^{\phi_+(x;t)}\psi_{-\frac{1}{2}}$ operator is one of the form $\Ket{\psi^*_{m-\frac{1}{2}}|-1}$ for some $0 \leq m \leq \lambda_1$.  (An analogous result holds in the other case.)  In the following theorem we compute the weight associated to $m = 0$; i.e., when the result of $n$ applications yields the vacuum. When we wish to use solely bra-kets made from partitions, the empty partition $\emptyset$ (which may be padded with $0$'s) is used to denote the vacuum state. 

\begin{theorem}\label{th:factorization.of.partition.function} Let $\lambda = (\lambda_1, \ldots, \lambda_n)$. Then with $\rho = (n,\cdots,2,1)$ we have
\begin{align*} 
\Braket{ \emptyset | \prod_{i=1}^n \left[ e^{\phi_+(x_i;t_i)} \psi_{-1/2} \right] | \lambda} &= \prod_{i=1}^n \left((-1)^i x_i(t_i+1)\right)  \left[ \prod_{i < j} (x_i + t_jx_j) \right] s_{\lambda-\rho}(x_1,\cdots,x_n)\\
\Braket{ \emptyset | \prod_{i=1}^n \left[\psi_{\lambda_1+\frac{1}{2}} e^{\phi_-(x_{n-i+1};t_{n-i+1})} \right] | \lambda} &=\prod_{i=1}^n \left(x_i^{-\lambda_1}(t_i+1)\right) \left[\prod_{i<j} (x_i+t_ix_j)\right] s_{\lambda-\rho}(x_1,\cdots,x_n),
\end{align*}
where the product in the bra-ket is taken so that $e^{\phi_+(x_n;t_n)}$ is rightmost in the first product, and $e^{\phi_-(x_1;t_1)}$ is rightmost in the second product.
\end{theorem}

Notice that the left side of  the first expression can be thought of as $$\prod_{\lambda^{(0)}}\prod_{\lambda^{(1)}}\cdots\prod_{\lambda^{(n-1)}} \Braket{\lambda^{(i-1)}|e^{\phi_+(x_i;t_i)}\psi_{-\frac{1}{2}}|\lambda^{(i)}}$$ as the $\lambda^{(i)}$ vary over all partitions with $i$ parts (and the only partition with $0$ parts is $\emptyset$). Theorem \ref{th:computing.weight.of.one.step} tells us that the only sequences of partitions $\lambda^{(n-1)},\cdots,\lambda^{(1)}$ that appear in this product are ones which satisfy the interleaving condition from the theorem.  (Of course, the same holds true when considering the second expression.)

\section{Pictorial representations for free-fermions\label{pictorial}}

Before moving towards the proofs of our main results, we pause briefly to carry out a few illustrative computations.  For this, it will be useful to have a pictorial representation for Fermion states.  We represent the vacuum state $|\emptyset\rangle$ with a sea of particles (black dots) occupying each negative half-integer position:
$$ \begin{tikzpicture}
\node[shape=circle,draw,fill=black!50] (a) at (-5/2,0) {};
\node[shape=circle,draw,fill=black!50] (b) at (-3/2,0) {};
\node[shape=circle,draw,fill=black!50] (c) at (-1/2,0) {};
\node[shape=circle,draw] (d) at (1/2,0) {};
\node[shape=circle,draw] (e) at (3/2,0) {};
\node[shape=circle,draw] (f) at (5/2,0) {};

\node[] at (-5,0) {$\Ket{\emptyset} =$};

\node [label=below:$0$] at (0,0) {};
\node [label=below:$-1$] at (-1,0) {};
\node [label=below:$-2$] at (-2,0) {};
\node [label=below:$-3$] at (-3,0) {};
\node [label=left:$\cdots$] at (-3.1,0) {};

\node [label=below:$1$] at (1,0) {};
\node [label=below:$2$] at (2,0) {};
\node [label=below:$3$] at (3,0) {};
\node [label=right:$\cdots$] at (3.1,0) {};

\draw [-] (-3.25,0) -- (a) -- (b) -- (c) -- (d) -- (e) -- (f) -- (3.25,0);
\draw [-] (-3,.1) -- (-3,-.1);
\draw [-] (-2,.1) -- (-2,-.1);
\draw [-] (-1,.1) -- (-1,-.1);
\draw [-] (0,.1) -- (0,-.1);
\draw [-] (1,.1) -- (1,-.1);
\draw [-] (2,.1) -- (2,-.1);
\draw [-] (3,.1) -- (3,-.1);
\end{tikzpicture}
$$
Given any integer $\ell$, the state $\Ket{\ell}$ is depicted simply by
$$ \begin{tikzpicture}
\node[shape=circle,draw,fill=black!50] (a) at (-5/2,0) {};
\node[shape=circle,draw,fill=black!50] (b) at (-3/2,0) {};
\node[shape=circle,draw,fill=black!50] (c) at (-1/2,0) {};
\node[shape=circle,draw] (d) at (1/2,0) {};
\node[shape=circle,draw] (e) at (3/2,0) {};
\node[shape=circle,draw] (f) at (5/2,0) {};

\node[] at (-5,0) {$\Ket{\ell} =$};

\node [label=below:$\ell$] at (0,0) {};
\node [label=left:$\cdots$] at (-3.1,0) {};
\node [label=right:$\cdots$] at (3.1,0) {};

\draw [-] (-3.25,0) -- (a) -- (b) -- (c) -- (d) -- (e) -- (f) -- (3.25,0);
\draw [-] (-3,.1) -- (-3,-.1);
\draw [-] (-2,.1) -- (-2,-.1);
\draw [-] (-1,.1) -- (-1,-.1);
\draw [-] (0,.1) -- (0,-.1);
\draw [-] (1,.1) -- (1,-.1);
\draw [-] (2,.1) -- (2,-.1);
\draw [-] (3,.1) -- (3,-.1);
\end{tikzpicture}
$$
and for a strictly decreasing partition $\lambda$, one depicts $\Ket{\lambda}$ by taking the vacuum state and creating particles at the positions $\lambda_i-\frac{1}{2}$.  These diagrams are useful, but caution is required. The state $\Ket{\ell}$ refers to a precise order of creation or annihilation operators applied to the vacuum as in (\ref{ellstate}). If one wants to perform a creation or annihilation operation and express the result as a linear combination of $\Ket{\lambda}$, one must keep track of the associated signs from commutation relations. This is achieved by counting the number of particles to the right of the position where the creation or annihilation is taking place; the resulting sign is simply $-1$ to this power.

For the purposes of illustration, we examine the action of our shifted Hamiltonians on $\Ket{\lambda}$ with $\lambda = (5,3,2)$.  
Recall that we assigned variables corresponding to each non-zero part in $\lambda$ and factored the Hamiltonian as a product of operators $e^{\phi_+(x_i;t_i)}$ or $e^{\phi_-(x_i;t_i)}$ for each pair of variables $(x_i; t_i)$ as in Definition~\ref{superhamdef}. For the $e^{\phi_+(x_3; t_3)}$ operator, we compute \begin{equation}\label{eq:single.evolution}\sum_{k=0}^\infty \frac{1}{k!} \left(\sum_{q \geq 1} \frac{1}{q} (1-(-t_3)^{q})x_3^qJ_q\right)^k \psi_{-\frac{1}{2}}\psi^*_{5-\frac{1}{2}}\psi^*_{3-\frac{1}{2}}\psi^*_{2-\frac{1}{2}}\Ket{\emptyset}.\end{equation} The action of operator $J_q$ gives all ways of moving a single particle $q$ units to the left. So we must analyze which moves are non-zero on the state $\psi_{-\frac{1}{2}} \Ket{(5,3,2)}$:
$$ \begin{tikzpicture}
\node[shape=circle,draw,fill=black!50] (aa) at (-11/2,0) {};
\node[shape=circle,draw,fill=black!50] (bb) at (-9/2,0) {};
\node[shape=circle,draw] (cc) at (-7/2,0) {};
\node[shape=circle,draw] (a) at (-5/2,0) {};
\node[shape=circle,draw,fill=black!50] (b) at (-3/2,0) {};
\node[shape=circle,draw,fill=black!50] (c) at (-1/2,0) {};
\node[shape=circle,draw] (d) at (1/2,0) {};
\node[shape=circle,draw,fill=black!50] (e) at (3/2,0) {};
\node[shape=circle,draw] (f) at (5/2,0) {};
\node[shape=circle,draw] (g) at (7/2,0) {};

\node [label=below:] at (0,0) {};
\node [label=below:] at (-1,0) {};
\node [label=below:] at (-2,0) {};
\node [label=below:$0$] at (-3,0) {};
\node [label=below:] at (-4,0) {};
\node [label=below:] at (-5,0) {};
\node [label=below:] at (-6,0) {};
\node [label=left:$\cdots$] at (-6.1,0) {};

\node [label=below:] at (1,0) {};
\node [label=below:] at (2,0) {};
\node [label=below:] at (3,0) {};
\node [label=below:] at (4,0) {};
\node [label=right:$\cdots$] at (4.1,0) {};

\draw [-] (-6.25,0)--(aa)--(bb)--(cc)-- (a) -- (b) -- (c) -- (d) -- (e) -- (f) --(g) -- (4.25,0);
\draw [-] (-6,.1) -- (-6,-.1);
\draw [-] (-5,.1) -- (-5,-.1);
\draw [-] (-4,.1) -- (-4,-.1);
\draw [-] (-3,.1) -- (-3,-.1);
\draw [-] (-2,.1) -- (-2,-.1);
\draw [-] (-1,.1) -- (-1,-.1);
\draw [-] (0,.1) -- (0,-.1);
\draw [-] (1,.1) -- (1,-.1);
\draw [-] (2,.1) -- (2,-.1);
\draw [-] (3,.1) -- (3,-.1);
\draw [-] (4,.1) -- (4,-.1);
\end{tikzpicture}
$$
Our relations (\ref{pluscomm}) imply that creating a particle where one already exists annihilates the state. Thus in our example, the action by $J_q$ is 0 for any $q \geq 6$.  In fact, the following are the products of operators which yield nonzero states: $J_1$, $J_2$, $J_3$, $J_4$, $J_5$, $J_1J_1$, $J_1J_3$, $J_2J_1$, $J_2J_2$, $J_2J_3$, $J_2J_5$, $J_3J_1$, $J_3J_2$, $J_3J_4$, $J_4J_2$, $J_4J_3$, $J_5J1$, $J_5J_2$, $J_1J_1J_1$, $J_1J_1J_2$, $J_1J_1J_3$, $J_1J_1J_5$, $J_1J_2J_1$, $J_1J_2J_2$, $J_1J_2J_3$, $J_1J_2J_4$, $J_1J_3J_1$, $J_1J_3J_2$, $J_1J_3J_3$, $J_1J_4J_1$, $J_1J_4J_2$, $J_2J_1J_1$, $J_2J_1J_2$, $J_2J_1J_3$, $J_2J_1J_4$, $J_2J_2J_1$, $J_2J_2J_2$, $J_2J_2J_3$, $J_2J_3J_1$, $J_2J_3J_2$, $J_2J_4J_1$, $J_3J_1J_1$, $J_3J_1J_2$, $J_3J_1J_3$, $J_3J_2J_2$, $J_4J_1J_1$, $J_4J_1J_2$, $J_4J_2J_1$, $J_5J_1J_1$.  Obviously there are far too many here to compute explicitly, so we'll content ourselves with a slightly different question: what is the coefficient of $\Ket{(4,3)}$ in the ket appearing in equation (\ref{eq:single.evolution})?  That is, what is the value of 
$$\Braket{(4,3)|e^{\phi_+(x_3;t_3)}\psi_{-\frac{1}{2}}|(5,3,2)}?$$
Note that since the $J_q$ operators from $\phi_+$ only move particles to the left,  the particle which begins at position $5-\frac{1}{2}$ must move to $4-\frac{1}{2}$, the particle which begins at position $3-\frac{1}{2}$ must stay fixed, and the particle which begins at position $2-\frac{1}{2}$ must move to $-\frac{1}{2}$.  The possible migrations (coming from various summands in the operators $J_1J_1J_1$, $J_2J_1$ and $J_1J_2$) are given in Figure \ref{fig:sample.migrations}.

\begin{figure}[!ht]
\centering
\begin{subfigure}[!ht]{.48\textwidth}
\centering
\begin{tikzpicture}[auto,>=latex]
\node[shape=circle,draw] (cc1) at (-7/2,0) {};
\node[shape=circle,draw] (a1) at (-5/2,0) {};
\node[shape=circle,draw,fill=black!50] (b1) at (-3/2,0) {};
\node[shape=circle,draw,fill=black!50] (c1) at (-1/2,0) {};
\node[shape=circle,draw] (d1) at (1/2,0) {};
\node[shape=circle,draw,fill=black!50] (e1) at (3/2,0) {};

\node[shape=circle,draw] (cc2) at (-7/2,-1.5) {};
\node[shape=circle,draw] (a2) at (-5/2,-1.5) {};
\node[shape=circle,draw,fill=black!50] (b2) at (-3/2,-1.5) {};
\node[shape=circle,draw,fill=black!50] (c2) at (-1/2,-1.5) {};
\node[shape=circle,draw] (d2) at (1/2,-1.5) {};
\node[shape=circle,draw,fill=black!50] (e2) at (3/2,-1.5) {};

\node[shape=circle,draw] (cc3) at (-7/2,-3) {};
\node[shape=circle,draw] (a3) at (-5/2,-3) {};
\node[shape=circle,draw,fill=black!50] (b3) at (-3/2,-3) {};
\node[shape=circle,draw,fill=black!50] (c3) at (-1/2,-3) {};
\node[shape=circle,draw] (d3) at (1/2,-3) {};
\node[shape=circle,draw,fill=black!50] (e3) at (3/2,-3) {};

\node[shape=circle,draw] (cc4) at (-7/2,-4.5) {};
\node[shape=circle,draw] (a4) at (-5/2,-4.5) {};
\node[shape=circle,draw,fill=black!50] (b4) at (-3/2,-4.5) {};
\node[shape=circle,draw,fill=black!50] (c4) at (-1/2,-4.5) {};
\node[shape=circle,draw] (d4) at (1/2,-4.5) {};
\node[shape=circle,draw,fill=black!50] (e4) at (3/2,-4.5) {};

\node[shape=circle,draw] (cc5) at (-7/2,-6) {};
\node[shape=circle,draw] (a5) at (-5/2,-6) {};
\node[shape=circle,draw,fill=black!50] (b5) at (-3/2,-6) {};
\node[shape=circle,draw,fill=black!50] (c5) at (-1/2,-6) {};
\node[shape=circle,draw] (d5) at (1/2,-6) {};
\node[shape=circle,draw,fill=black!50] (e5) at (3/2,-6) {};

\foreach \j in {0,-1.5,...,-6}
	\node [label=below:$0$] at (-3,\j) {};


\draw [-] (-4.25,0)--(cc1)-- (a1) -- (b1) -- (c1) -- (d1) -- (e1) --  (2.25,0);
\draw [-] (-4.25,-1.5)--(cc2)-- (a2) -- (b2) -- (c2) -- (d2) -- (e2) --  (2.25,-1.5);
\draw [-] (-4.25,-3)--(cc3)-- (a3) -- (b3) -- (c3) -- (d3) -- (e3) --  (2.25,-3);
\draw [-] (-4.25,-4.5)--(cc4)-- (a4) -- (b4) -- (c4) -- (d4) -- (e4) --  (2.25,-4.5);
\draw [-] (-4.25,-6)--(cc5)-- (a5) -- (b5) -- (c5) -- (d5) -- (e5) -- (2.25,-6);

\foreach \j in {0,-1.5,...,-6}
	\foreach \i in {-4,-3,...,2}
		\draw [-] (\i,\j+.1) -- (\i,\j-.1);

\draw [->,dashed] (b1)   .. controls +(-1/4,2/4) and +(1/4,2/4) .. node[swap] {\tiny{1}} (a1);
\draw [->,dashed] (a1)   .. controls +(-1/4,2/4) and +(1/4,2/4) .. node[swap] {\tiny{2}} (cc1);
\draw [->,dashed] (e1)   .. controls +(-1/4,2/4) and +(1/4,2/4) .. node[swap] {\tiny{3}} (d1);

\draw [->,dashed] (b2)   .. controls +(-1/4,2/4) and +(1/4,2/4) .. node[swap] {\tiny{1}} (a2);
\draw [->,dashed] (a2)   .. controls +(-1/4,2/4) and +(1/4,2/4) .. node[swap] {\tiny{3}} (cc2);
\draw [->,dashed] (e2)   .. controls +(-1/4,2/4) and +(1/4,2/4) .. node[swap] {\tiny{2}} (d2);

\draw [->,dashed] (b3)   .. controls +(-1/4,2/4) and +(1/4,2/4) .. node[swap] {\tiny{2}} (a3);
\draw [->,dashed] (a3)   .. controls +(-1/4,2/4) and +(1/4,2/4) .. node[swap] {\tiny{3}} (cc3);
\draw [->,dashed] (e3)   .. controls +(-1/4,2/4) and +(1/4,2/4) .. node[swap] {\tiny{1}} (d3);

\draw [->,dashed] (b4)   .. controls +(-1/4,2/4) and +(1/4,2/4) .. node[swap] {\tiny{1}} (cc4);
\draw [->,dashed] (e4)   .. controls +(-1/4,2/4) and +(1/4,2/4) .. node[swap] {\tiny{2}} (d4);

\draw [->,dashed] (b5)   .. controls +(-1/4,2/4) and +(1/4,2/4) .. node[swap] {\tiny{2}} (cc5);
\draw [->,dashed] (e5)   .. controls +(-1/4,2/4) and +(1/4,2/4) .. node[swap] {\tiny{1}} (d5);
\end{tikzpicture}
\end{subfigure}
\begin{subfigure}[h]{.48\textwidth}
  \centering
\begin{tikzpicture}[auto,>=latex]
\node[shape=circle,draw] (a1) at (-5/2,0) {};
\node[shape=circle,draw,fill=black!50] (b1) at (-3/2,0) {};
\node[shape=circle,draw,fill=black!50] (c1) at (-1/2,0) {};
\node[shape=circle,draw] (d1) at (1/2,0) {};
\node[shape=circle,draw,fill=black!50] (e1) at (3/2,0) {};
\node[shape=circle,draw] (f1) at (5/2,0) {};

\node[shape=circle,draw] (a2) at (-5/2,-1.5) {};
\node[shape=circle,draw,fill=black!50] (b2) at (-3/2,-1.5) {};
\node[shape=circle,draw,fill=black!50] (c2) at (-1/2,-1.5) {};
\node[shape=circle,draw] (d2) at (1/2,-1.5) {};
\node[shape=circle,draw,fill=black!50] (e2) at (3/2,-1.5) {};
\node[shape=circle,draw] (f2) at (5/2,-1.5) {};

\node[shape=circle,draw] (a3) at (-5/2,-3) {};
\node[shape=circle,draw,fill=black!50] (b3) at (-3/2,-3) {};
\node[shape=circle,draw,fill=black!50] (c3) at (-1/2,-3) {};
\node[shape=circle,draw] (d3) at (1/2,-3) {};
\node[shape=circle,draw,fill=black!50] (e3) at (3/2,-3) {};
\node[shape=circle,draw] (f3) at (5/2,-3) {};

\node[shape=circle,draw] (a4) at (-5/2,-4.5) {};
\node[shape=circle,draw,fill=black!50] (b4) at (-3/2,-4.5) {};
\node[shape=circle,draw,fill=black!50] (c4) at (-1/2,-4.5) {};
\node[shape=circle,draw] (d4) at (1/2,-4.5) {};
\node[shape=circle,draw,fill=black!50] (e4) at (3/2,-4.5) {};
\node[shape=circle,draw] (f4) at (5/2,-4.5) {};

\node[shape=circle,draw] (a5) at (-5/2,-6) {};
\node[shape=circle,draw,fill=black!50] (b5) at (-3/2,-6) {};
\node[shape=circle,draw,fill=black!50] (c5) at (-1/2,-6) {};
\node[shape=circle,draw] (d5) at (1/2,-6) {};
\node[shape=circle,draw,fill=black!50] (e5) at (3/2,-6) {};
\node[shape=circle,draw] (f5) at (5/2,-6) {};

\foreach \j in {0,-1.5,...,-6}
	\node [label=below:$0$] at (-3,\j) {};


\draw [-] (-3.25,0)--(a1) -- (b1) -- (c1) -- (d1) -- (e1) -- (f1) -- (3.25,0);
\draw [-] (-3.25,-1.5)-- (a2) -- (b2) -- (c2) -- (d2) -- (e2) -- (f2) -- (3.25,-1.5);
\draw [-] (-3.25,-3)-- (a3) -- (b3) -- (c3) -- (d3) -- (e3) -- (f3) -- (3.25,-3);
\draw [-] (-3.25,-4.5)-- (a4) -- (b4) -- (c4) -- (d4) -- (e4) -- (f4) -- (3.25,-4.5);
\draw [-] (-3.25,-6)-- (a5) -- (b5) -- (c5) -- (d5) -- (e5) -- (f5) -- (3.25,-6);

\foreach \j in {0,-1.5,...,-6}
	\foreach \i in {-3,-2,...,3}
		\draw [-] (\i,\j+.1) -- (\i,\j-.1);

\draw [->,dashed] (b1)   .. controls +(1/4,2/4) and +(-1/4,2/4) .. node {\tiny{3}} (c1);
\draw [->,dashed] (c1)   .. controls +(1/4,2/4) and +(-1/4,2/4) .. node {\tiny{2}} (d1);
\draw [->,dashed] (e1)   .. controls +(1/4,2/4) and +(-1/4,2/4) .. node {\tiny{1}} (f1);

\draw [->,dashed] (b2)   .. controls +(1/4,2/4) and +(-1/4,2/4) .. node {\tiny{2}} (c2);
\draw [->,dashed] (c2)   .. controls +(1/4,2/4) and +(-1/4,2/4) .. node {\tiny{1}} (d2);
\draw [->,dashed] (e2)   .. controls +(1/4,2/4) and +(-1/4,2/4) .. node {\tiny{3}} (f2);

\draw [->,dashed] (b3)   .. controls +(1/4,2/4) and +(-1/4,2/4) .. node {\tiny{3}} (c3);
\draw [->,dashed] (c3)   .. controls +(1/4,2/4) and +(-1/4,2/4) .. node {\tiny{1}} (d3);
\draw [->,dashed] (e3)   .. controls +(1/4,2/4) and +(-1/4,2/4) .. node {\tiny{2}} (f3);

\draw [->,dashed] (b4)   .. controls +(1/4,2/4) and +(-1/4,2/4) .. node {\tiny{1}} (d4);
\draw [->,dashed] (e4)   .. controls +(1/4,2/4) and +(-1/4,2/4) .. node {\tiny{2}} (f4);

\draw [->,dashed] (b5)   .. controls +(1/4,2/4) and +(-1/4,2/4) .. node {\tiny{2}} (d5);
\draw [->,dashed] (e5)   .. controls +(1/4,2/4) and +(-1/4,2/4) .. node {\tiny{1}} (f5);
\end{tikzpicture}
\end{subfigure}
\caption{Leftward migrations from $\psi_{-\frac{1}{2}}(5,3,2)$ to $(4,3)$, and rightward migrations from $(5,3,2)$ to $(6,4,3)$.  Arrow labels indicate the order of movement.}\label{fig:sample.migrations}
\end{figure}
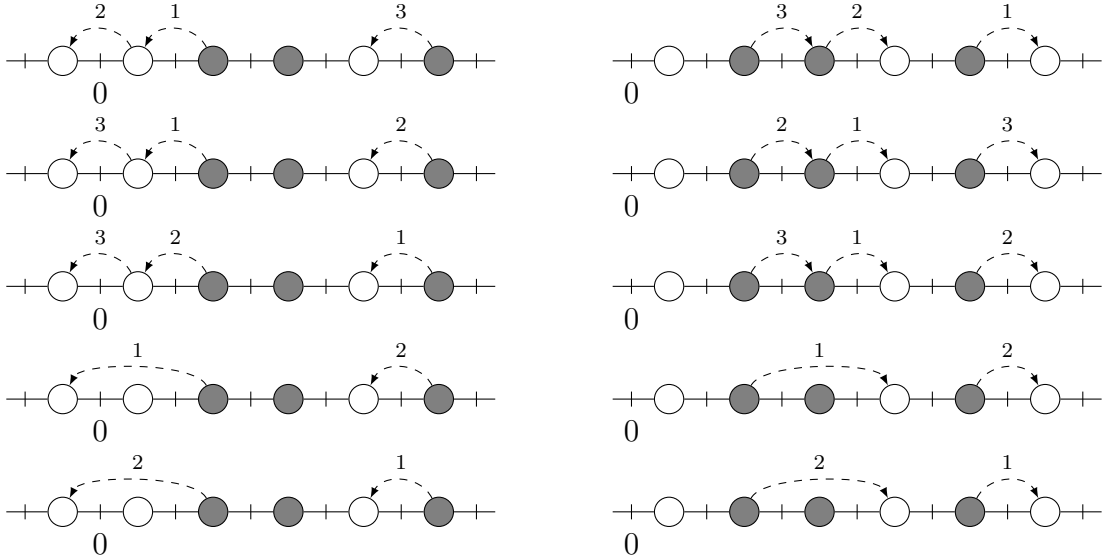

Analyzing the contribution from each of these and summing yields a coefficient of 
$$-3\frac{1}{3!}(1+t_3)^3x_3^3 - 2\frac{1}{2!}\frac{1}{2}(1-t_3^2)(1+t_3)x_3^3 = -x_3^3(1+t_3)^2.$$

Note that for the $e^{\phi_-}$ Hamiltonian, the action of the associated $J_{q}$ on $\Ket{\lambda}$ moves particles to the right.  Hence analyzing all possible rightward motions is even more intractable, since there is a sea of antiparticles on the right which allows for particles to move as far as they like.  We will compute $$\Braket{(4,3)|\psi_{6-\frac{1}{2}} e^{\phi_-(x_1;t_1)}|(5,3,2)}.$$  Since our operator finishes by annihilating at position $6-\frac{1}{2}$, it suffices for us to determine those rightward migrations that carry $\Ket{(5,3,2)}$ to $\Ket{(6,4,3)}$.  The possible migrations are given in Figure \ref{fig:sample.migrations}.  Analyzing the contribution from each yields
$$3\frac{1}{3!}(1+t_3)^3x_1^{-3} - 2\frac{1}{2!}\frac{1}{2}(1-t_1^2)(1+t_1)x_1^{-3} = x_1^{-3}(1+t_1)^2t_1.$$

While none of our results rely on this pictorial representation of fermions and their images under Hamiltonian operators, we find it useful to see directly how the Hamiltonians are behaving and to have tools for explicit computation of special cases of the results above.

\section{A connection to ice models}

The value of the bra-ket from Theorem \ref{th:computing.weight.of.one.step} has previously appeared in a generating function identity of Tokuyama~\cite{tokuyama} phrased in the language of ``strict'' Gelfand-Tsetlin patterns.  A Gelfand-Tsetlin pattern (of size $n$) is a sequence of partitions $\lambda^{(i)} = (a_{i,1} \geq \cdots \geq a_{i,i} \geq 0)$ (for $1 \leq i \leq n$) that satisfy the interleaving condition $$a_{i+1,j} \geq a_{i,j} \geq a_{i+1,j+1}.$$  Such a pattern is called strict if each $\lambda^{(i)}$ is strictly decreasing.  A Gelfand-Tsetlin pattern is often depicted as a triangular array
$$\begin{array}{ccccccccc}
a_{n,1} &            & a_{n,2} & \cdots  & a_{n,n-1} &              & a_{n,n} \\[5pt]
        &  a_{n-1,1} &         & \cdots  &           &  a_{n-1,n-1} &         \\
        &            & \ddots  &         & \iddots   &              &         \\
        &            &         & a_{1,1} &           &              &
\end{array}
$$
and so the $\lambda^{(i)}$ are often called the rows of the pattern. The set of all Gelfand-Tsetlin patterns with fixed top row arise naturally in representation theory, as they parametrize bases for highest weight representations of $GL_n(\mathbb{C})$ according to the multiplicity-free branching rules from $GL_n$ to $GL_{n-1}$.

Strict Gelfand-Tsetlin patterns are perhaps less natural, but it was noticed by Hamel and King \cite{hamel-king} that they are in bijection with admissible states
of certain two-dimensional lattice models (or ``ice'' models). This bijection was exploited in \cite{bbf-ice} to give a proof of Tokuyama's formula using the Yang-Baxter equation.
We now describe this bijection more precisely. Consider the set of balanced directed graphs with $\lambda_1$-columns (labeled from left to right as $\lambda_1$ through $1$) and $n$-rows (labeled from top to bottom as $n$ through $1$), subject to the  following boundary conditions: inward pointing boundary edges along each row, downward pointing boundary edges along the bottom boundary, and upward pointing edges at the top boundary of a column if and only if the column index corresponds to a part of $\lambda$.  Such a graph is called an admissible state of ice, and we write $\mathfrak{A}^\lambda$ for the set of all such admissible states. To create a state of ice from a pattern $\lambda^{(n)},\cdots,\lambda^{(1)}$, one labels the northern edge of the vertex at $(i,j)$ upward if and only if $j$ is a part of $\lambda^{(i)}$; this determines the direction along any vertical edge, and the horizontal edges are filled in inductively so as to produce a state of $\mathfrak{A}^\lambda$. For example, under this bijection we make the following identification:

\begin{equation} \label{icegtbijection} \begin{array}{ccc}
     \vcenter{\hbox to 2in{
     \begin{tikzpicture}[scale=.8, every node/.style={transform shape}]

\node [label=left:$3$] at (0,2) {};
\node [label=left:$2$] at (0,1) {};
\node [label=left:$1$] at (0,0) {};

\node [label=above:$5$] at (0.5,2.5) {};
\node [label=above:$4$] at (1.5,2.5) {};
\node [label=above:$3$] at (2.5,2.5) {};
\node [label=above:$2$] at (3.5,2.5) {};
\node [label=above:$1$] at (4.5,2.5) {};

\draw [>-] (0,2) -- (1,2);
\draw [>-] (0,1) -- (1,1);
\draw [>-] (0,0) -- (1,0);

\draw [<-] (0.5,2.5) -- (0.5,1.5);
\draw [>->] (0.5,1.5) -- (0.5,.5);
\draw [->] (0.5,0.5) -- (0.5,-0.5);

\draw [<-] (1,2) -- (2,2);
\draw [>-] (1,1) -- (2,1);
\draw [>-] (1,0) -- (2,0);

\draw [>-] (1.5,2.5) -- (1.5,1.5);
\draw [<->] (1.5,1.5) -- (1.5,.5);
\draw [->] (1.5,0.5) -- (1.5,-0.5);

\draw [>-] (2,2) -- (3,2);
\draw [<-] (2,1) -- (3,1);
\draw [>-] (2,0) -- (3,0);

\draw [<-] (2.5,2.5) -- (2.5,1.5);
\draw [<-<] (2.5,1.5) -- (2.5,.5);
\draw [->] (2.5,0.5) -- (2.5,-0.5);

\draw [>-] (3,2) -- (4,2);
\draw [<-] (3,1) -- (4,1);
\draw [<-] (3,0) -- (4,0);

\draw [<-] (3.5,2.5) -- (3.5,1.5);
\draw [>->] (3.5,1.5) -- (3.5,.5);
\draw [->] (3.5,0.5) -- (3.5,-0.5);

\draw [<-<] (4,2) -- (5,2);
\draw [<-<] (4,1) -- (5,1);
\draw [<-<] (4,0) -- (5,0);

\draw [>-] (4.5,2.5) -- (4.5,1.5);
\draw [>->] (4.5,1.5) -- (4.5,.5);
\draw [->] (4.5,0.5) -- (4.5,-0.5);

\end{tikzpicture} }} & \longleftrightarrow & \begin{array}{ccccc} 5 & & 3 & & 2 \\ & 4 & & 3 & \\ & & 3 & &  \end{array}  \end{array} \end{equation}

One can associate a weight to a given state of ice as follows.  For each of the six possible vertex arrangements, one assigns a local weight. The assignment of the local weight is according to the six possible configurations of adjacent edges to the vertex, and is further allowed to depend on the row in which the vertex sits. For a state $A \in \mathfrak{A}^\lambda$, we assign $\text{wt}(A)$ to be the product of the local weights at all of the vertices in $A$.  The so-called partition function for $\mathfrak{A}^\lambda$ (relative to the given choice of weights) is then $$\mathcal{Z}(\mathfrak{A}^\lambda) = \sum_{A \in \mathfrak{A}^\lambda} \text{wt}(A).$$

In \cite{bbf-ice}, two weight assignments -- denoted $\Delta$ and $\Gamma$ -- were defined, and shown to satisfy Yang-Baxter equations and have partition functions expressible as Schur polynomials multiplied by a deformed denominator. In particular, the weight assignment $\Gamma$ gives Boltzmann weights for ice that correspond bijectively to those in Tokuyama's generating function over strict patterns (using the above bijection, but with rows labeled in ascending order). The $\Delta$ weights correspond bijectively to a generating function over strict patterns that is --- in some sense --- dual to Tokuyama's: it replaces an exponent from Tokuyama's original formula expressed in terms of ``left-leaning" entries with one in terms of ``right-leaning" entries, and also lists rows in descending order. Both schemes are addressed in \cite{bbf-ice}. The two weight assignments $\Gamma$ and $\Delta$ have been reproduced in Figure \ref{fig:six.vertex.model}; the spectral parameters have been set up to match the row-numbering conventions for the two weight schemes (relative to our convention of numbering rows in descending order).

\begin{figure}[!ht]
\centering
\begin{tabular}
{|c|c|c|c|c|c|c|}
\hline
&
$\begin{tikzpicture}[>=angle 90,shorten <=-1.5pt,shorten >=-1.5pt]
\node [label=left:$j$] at (0.5,1) {};
\node [label=right:$ $] at (1.5,1) {};
\node [label=above:$ $] at (1,1.5) {};
\node [label=below:$ $] at (1,0.5) {};
\draw [>->] (0.5,1) -- (1.5,1);
\draw [>->] (1,1.5) -- (1,0.5);
\end{tikzpicture}$
&
$\begin{tikzpicture}[>=angle 90,shorten <=-1.5pt,shorten >=-1.5pt]
\node [label=left:$j$] at (0.5,1) {};
\node [label=right:$ $] at (1.5,1) {};
\node [label=above:$ $] at (1,1.5) {};
\node [label=below:$ $] at (1,0.5) {};
\draw [<-<] (0.5,1) -- (1.5,1);
\draw [<-<] (1,1.5) -- (1,0.5);
\end{tikzpicture}$
&
$\begin{tikzpicture}[>=angle 90,shorten <=-1.5pt,shorten >=-1.5pt]
\node [label=left:$j$] at (0.5,1) {};
\node [label=right:$ $] at (1.5,1) {};
\node [label=above:$ $] at (1,1.5) {};
\node [label=below:$ $] at (1,0.5) {};
\draw [>->] (0.5,1) -- (1.5,1);
\draw [<-<] (1,1.5) -- (1,0.5);
\end{tikzpicture}$
&
$\begin{tikzpicture}[>=angle 90,shorten <=-1.5pt,shorten >=-1.5pt]
\node [label=left:$j$] at (0.5,1) {};
\node [label=right:$ $] at (1.5,1) {};
\node [label=above:$ $] at (1,1.5) {};
\node [label=below:$ $] at (1,0.5) {};
\draw [<-<] (0.5,1) -- (1.5,1);
\draw [>->] (1,1.5) -- (1,0.5);
\end{tikzpicture}$
&
$\begin{tikzpicture}[>=angle 90,shorten <=-1.5pt,shorten >=-1.5pt]
\node [label=left:$j$] at (0.5,1) {};
\node [label=right:$ $] at (1.5,1) {};
\node [label=above:$ $] at (1,1.5) {};
\node [label=below:$ $] at (1,0.5) {};
\draw [<->] (0.5,1) -- (1.5,1);
\draw [>-<] (1,1.5) -- (1,0.5);
\end{tikzpicture}$
&
$\begin{tikzpicture}[>=angle 90,shorten <=-1.5pt,shorten >=-1.5pt]
\node [label=left:$j$] at (0.5,1) {};
\node [label=right:$ $] at (1.5,1) {};
\node [label=above:$ $] at (1,1.5) {};
\node [label=below:$ $] at (1,0.5) {};
\draw [>-<] (0.5,1) -- (1.5,1);
\draw [<->] (1,1.5) -- (1,0.5);
\end{tikzpicture}$
\\
\hline
$\Delta$ & $1$&$t_jx_j$&$1$&$x_j$&$x_j(t_j+1)$&$1$\\ \hline
$\Gamma$ & $1$&$x_{n-j+1}$&$t_{n-j+1}$&$x_{n-j+1}$&$x_{n-j+1}(t_{n-j+1}+1)$&$1$\\ \hline
\end{tabular}
\caption{$\Delta$ and $\Gamma$ weighting schemes from \cite{bbf-ice}. The index $j$ denotes the row number in which the
vertex appears.}\label{fig:six.vertex.model}
\end{figure}

We will write $\mathcal{Z}_\Delta(\mathfrak{A}^\lambda)$ for the partition function associated to the $\Delta$-weighting scheme, and likewise for $\Gamma$.

The following result thus records that our Hamiltonians $\prod_{i=1}^n \psi_{\lambda_1+\frac{1}{2}}e^{\phi_-(x_{n-i+1};t_{n-1+1})}$ and $\prod_{i=1}^n e^{\phi_+(x_i;t_i)}\psi_{-\frac{1}{2}}$ from Definition~\ref{superhamdef} and Theorem \ref{th:computing.weight.of.one.step} may be used to recover (up to a simple factor) the row-by-row statistics of Tokuyama's generating function and its dual on patterns; equivalently, our Hamiltonians capture the row-by-row contribution to the Boltzmann weight of ice with weight schemes $\Gamma$ and $\Delta$ of \cite{bbf-ice}.
Recall that, as above, we defined $\lambda^{(j)}$ to be the partition that records the columns along row $j$ whose northern edge is decorated with an ``up" arrow (and with $\lambda^{(0)} = \emptyset$).

\begin{corollary}\label{cor:rectangular.ice.and.fermions}
Let $A$ be a state of ice which corresponds to a pattern $\mathcal{T}$ with top row $\lambda := \lambda^{(n)} = (\lambda_1>\cdots>\lambda_n)$, and denote the $\Delta$-weight of the $i$th row of $A$ by $\text{wt}_{\Delta}(A;i)$ and the $\Gamma$-weight of the $i$th row of $A$ by $\text{wt}_\Gamma(A;i)$.  Then
\begin{align*}
(-1)^i(t_i+1)x_i\cdot\text{wt}_\Delta(A;i) &= \Braket{ \lambda^{(i-1)}|e^{\phi_+(x_i;t_i)}\psi_{-\frac{1}{2}}|\lambda^{(i)}}\\
(t_{n-i+1}+1)x_{n-i+1}^{-\lambda_1}\cdot\text{wt}_\Gamma(A;i) &= \Braket{ \lambda^{(i-1)}|\psi_{\lambda_1+\frac{1}{2}}e^{\phi_+(x_{n-i+1};t_{n-i+1})}|\lambda^{(i)}}.
\end{align*}
In particular,
\begin{align*}
\left(\prod_{i=1}^n (-1)^i (t_i+1)x_i\right)\text{wt}_\Delta(A) &= \prod_{i=1}^n\Braket{\lambda^{(i-1)}|e^{\phi_+(x_i;t_i)}\psi_{-\frac{1}{2}}|\lambda^{(i)}}\\
\left(\prod_{i=1}^n (t_{n-i+1}+1)x_{n-i+1}^{-\lambda_1}\right)\text{wt}_\Gamma(A) &= \prod_{i=1}^n\Braket{\lambda^{(i-1)}|e^{\phi_+(x_{n-i+1};t_{n-i+1})}\psi_{-\frac{1}{2}}|\lambda^{(i)}}.
\end{align*}
\end{corollary}

\begin{proof}
The $\Gamma$ weights in \cite{bbf-ice} match up with Tokuyama's weighting schemes for patterns on a row-by-row basis (see \cite[Th.~6]{bbf-ice}).  Now we simply observe that Theorem \ref{th:computing.weight.of.one.step} gives the value of $\Braket{\mu|\psi_{\lambda_1+\frac{1}{2}}e^{\phi_-(x_{n-i+1};t_{n-i+1})}|\lambda}$ as this same weight, up to the given factor.  (It's useful to note that since Tokyuama and \cite{bbf-ice} list their columns starting at $0$ instead of $1$, the quantity $|\lambda^{(i)}|-|\lambda^{(i-1)}|$ that we calculate is larger than the related quantity calculated by Tokuyama; hence we are left with the additional factor of $x_i^{-(\lambda_1+1)+1}$, as opposed to $x_i^{-(\lambda_1+1)}$.)   A similar result holds for the $\Delta$ weights and the $e^{\phi_+(x;t)}\psi_{-\frac{1}{2}}$ operator.
\end{proof}

\begin{example*}
Consider the row of ice corresponding to $(5,3,2), (4,3)$; this is shown in Figure \ref{fig:sample.ice}.  The $\Delta$-weight of the row is $x_3^2(t_3+1)$ and the corresponding bra-ket involving $\lambda^{(3)} = (5,3,2)$ and $\lambda^{(2-1)} = (4,3)$ was calculated at the end of Section~\ref{pictorial}, resulting in $-x_3^3(1+t_3)^2$. The reader may similarly verify that  its $\Gamma$-weight (where we take $n=3$, assuming that $(5,3,2)$ is the ``top" of the ice) is $x_1^2(t_1+1)t_1$, and the corresponding bra-ket calculated in Section~\ref{pictorial} is $x_1^{-3}(1+t_1)^2t_1$ agreeing with the second case of the above Corollary.  

\begin{figure}[!ht]
\begin{tikzpicture}[>=angle 90,shorten <=-1.5pt,shorten >=-1.5pt]

\node [label=left:$3$] at (0,2) {};

\node [label=above:$5$] at (0.5,2.5) {};
\node [label=above:$4$] at (1.5,2.5) {};
\node [label=above:$3$] at (2.5,2.5) {};
\node [label=above:$2$] at (3.5,2.5) {};
\node [label=above:$1$] at (4.5,2.5) {};

\draw [>-] (0,2) -- (1,2);

\draw [<->] (0.5,2.5) -- (0.5,1.5);

\draw [<-] (1,2) -- (2,2);

\draw [>-<] (1.5,2.5) -- (1.5,1.5);

\draw [>-] (2,2) -- (3,2);

\draw [<-<] (2.5,2.5) -- (2.5,1.5);

\draw [>-] (3,2) -- (4,2);

\draw [<->] (3.5,2.5) -- (3.5,1.5);

\draw [<-<] (4,2) -- (5,2);

\draw [>->] (4.5,2.5) -- (4.5,1.5);

\end{tikzpicture}
\caption{The row of ice corresponding to $(5,3,2), (4,3)$}
\label{fig:sample.ice}
\end{figure}
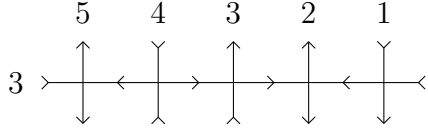
\end{example*}


Note that since all states have the same bottom-boundary (all arrows pointed down), one can interpret this as saying that the ``final row" of any Gelfand-Tsetlin pattern is the empty set.  This agrees with our previous notational choice of $\lambda^{(0)} = \emptyset$.  From our corollary, it therefore follows that if $\lambda$ is a partition with strictly decreasing parts, then the partition function $\mathcal{Z}_\Delta(\mathfrak{A}^\lambda) = \sum_A \text{wt}_\Delta(A)$ satisfies
$$\left(\prod_{i=1}^n (-1)^i (t_i+1)x_i\right) \mathcal{Z}_\Delta(\mathfrak{A}^\lambda) = \Braket{ \emptyset| \prod_{i=1}^n e^{\phi_+(x_i;t_i)}\psi_{-\frac{1}{2}} |\lambda}$$ and likewise for $\mathcal{Z}_\Gamma(\mathfrak{A}^\lambda)$. The left-hand side was already evaluated in \cite[Thm.~9]{bbf-ice}, and we could use that result to give a proof of Theorem \ref{th:factorization.of.partition.function} at this point.  In the next section, we will instead give a different proof of the evaluation of the right-hand side above using commutation relations for Hamiltonians.

\begin{remark*} The weighting schemes $\Gamma$ and $\Delta$ are just special cases of Boltzmann weights satisfying a quadratic relation
known as the free-fermion point of the six-vertex model. In \cite{bbf-ice}, it is shown that all partition functions using these weights are expressible as
Schur functions. However, the weights $\Gamma$ and $\Delta$ are the only choices which seem to be expressible in terms of Hamiltonians.
\end{remark*}

\section{Proofs of Theorems~\ref{th:computing.weight.of.one.step} and \ref{th:factorization.of.partition.function}\label{proofsofthms}}

In order to prove Theorems \ref{th:computing.weight.of.one.step} and \ref{th:factorization.of.partition.function}, we first determine the relation between the Hamiltonian action $e^{H[s(t)]}$ and the creation and annihilation operators.

\begin{proposition}\label{commutation} Let $H[s(t)]$ be given by
$$ H[s(t)] := \sum_{q \geq 1} s_q(t) J_q. $$
Then
\begin{align*} e^{H[s(t)]} \psi(z) e^{-H[s(t)]} &= e^{-\sum_{q \geq 1} s_q(t) z^q} \psi(z)\\[5pt]
e^{H[s(t)]} \psi^\ast(z) e^{-H[s(t)]} &= e^{\sum_{q \geq 1} s_q(t) z^q} \psi^\ast(z).
\end{align*}
In particular, the special case $n=1$ in the definition of $s_q^{(n)}(\boldsymbol{x} ; t)$ gives:
$$ e^{\phi_+(x;t)} \psi^{[\ast]}(z) e^{-\phi_+(x;t)} = \text{exp} \left( \pm \sum_{q \geq 1} \frac{1}{q} (1+(-1)^{q+1} t^q) (xz)^q \right) \psi^{[\ast]}(z). $$
\end{proposition}

\begin{proof}
This proof follows similarly to the case of $t=0$ sketched in \cite{Zinn-Justin}. It is not hard to check that $J_q \psi(z) - \psi(z) J_q = -z^q \psi(z)$; one applies the commutation relations to $J_q\psi_{-k} z^{k+\frac{1}{2}}$ and notices the relation $[\psi_j,\psi_i^*] = \delta_{ij}$ gives a term not found in $\psi_{-k}z^{k+\frac{1}{2}}J_q$; then sum over $k$.  This immediately implies (multiplying by $s_q(t)$ and summing over $q$) that
$$ H[s(t)] \psi(z) - \psi(z) H[s(t)] = - \sum_{q \geq 1} s_q(t) z^q \psi(z). $$
As Zinn-Justin \cite{Zinn-Justin} notes, it now remains to exponentiate both sides and analyze the result.

First note that 
$$ e^{\psi(z)} = \sum_{k=0}^\infty \frac{\psi(z)^k}{k!} = 1 + \psi(z) $$
because $[\psi_r, \psi_s]_+ := \psi_r \psi_s + \psi_s \psi_r = 0$ for all $r, s$ (even $r=s$) which implies that $\psi^2(z)=0$ and thus all higher terms in the exponential power series vanish.

Since
$$ e^X e^Y = e^{Y+[X,Y]+\frac{1}{2!} [X,[X,Y]] + \cdots} e^X, $$
then
$$ e^{H[s(t)] \psi(z)} = e^{\psi(z) + [H[s(t)],\psi(z)] + \cdots} e^{H[s(t)]} $$
But $[H[s(t)],\psi(z)] = f(z; t) \psi(z)$ where $f(z; t) = - \sum_{q \geq 1} s_q(t) z^q$, which commutes with any operator on fermions as a formal element
of the coefficient ring $\mathbb{C}[z, t]$. Moreover, the nested commutators will give rise to powers of $f(z;t)$ multiplying $\psi(z)$ since 
$$ [H[s(t)], f(z;t)\psi(z)]= f(z;t)^2 \psi(z), \; \text{etc.} $$ 
Thus the right-hand side becomes:
$$ e^{\psi(z) + [H[s(t)],\psi(z)] + \cdots} e^{H[s(t)]} = e^{e^{f(z;t)}\psi(z)} e^{H[s(t)]} = e^{e^{f(z;t)}\psi(z)} e^{H[s(t)]} = e^{e^{f(z;t)} \psi(z)} e^{H[s(t)]}, $$
which then implies
$$ e^{H[s(t)]} \psi(z) e^{-H[s(t)]} = e^{H[s(t)]} (e^{\psi(z)}-1) e^{-H[s(t)]} = e^{e^{f(z;t)} \psi(z)} - 1, $$
but again $e^{e^{f(z;t)} \psi(z)}$ is just $1 + e^{f(z;t)} \psi(z)$ and so this now gives:
$$ e^{H[s(t)]} \psi(z) e^{-H[s(t)]} = e^{e^{f(z;t)} \psi(z)} - 1 = e^{f(z;t)} \psi(z) = e^{- \sum_{q \geq 1} s_q(t) z^q} \psi(z), $$
as desired. The proof of the analogous formula for $\psi^\ast(z)$ follows similarly.
\end{proof}


In light of the previous result, it will be useful to have an expression for the quantity $e^{\sum_{q \geq 1} s_q(t) \sum_{i=1}^n z_i^q}$ which appears when commuting $e^{H[s(t)]}$ past $\psi(z)$.
\begin{lemma} \label{sumtheseries}
$$ e^{\sum_{q \geq 1} s_q(t) \sum_{i=1}^n z_i^q} = \prod_{1 \leq i, j \leq n} \frac{(1+t z_i x_j)}{(1-z_i x_j)} $$
\end{lemma}

\begin{proof} Recall that 
$$ s_q^{(n)}(\boldsymbol{x};t) = (1 + (-1)^{q+1} t^q) s_q^{(n)}(\boldsymbol{x}) = \frac{(1 + (-1)^{q+1} t^q)}{q} \sum_{j=1}^n x_j^q . $$
The result follows easily from the case $n=1$. There we sum the series
$$ \sum_{q \geq 1} \frac{1 + (-1)^{q+1} t^q}{q} (z x)^q = \log \left( \frac{1+t z x}{1 - z x} \right), $$
and exponentiating gives the result. 
\end{proof}



\begin{proof}[Proof of Theorem~\ref{th:computing.weight.of.one.step}] 
Recall that the state $\Ket{\lambda}$ is defined by
\begin{align*} \Ket{\lambda} & = \psi_{\lambda_1  - \frac{1}{2}}^\ast  \psi_{\lambda_2 - \frac{1}{2}}^\ast \cdots \psi_{\lambda_n - \frac{1}{2}}^\ast \Ket{\emptyset} \\
& = \psi^\ast(z_1) \psi^\ast(z_2) \cdots \psi^\ast(z_n) \Ket{ \emptyset } \, \Big\vert_{\boldsymbol{z}^{\lambda}}
\end{align*}
where the notation $\vert_{\boldsymbol{z}^{\lambda}}$ means we take the coefficient of $\boldsymbol{z}^{\lambda} = z_1^{\lambda_1} \cdots z_n^{\lambda_n}$. 
Thus
$$ \psi_{-1/2} \Ket{ \lambda } = (-1)^n \left[ \psi^\ast(z_1) \psi^\ast(z_2) \cdots \psi^\ast(z_n) \psi_{-\frac{1}{2}}\Ket{ \emptyset } \, \Big\vert_{\boldsymbol{z}^{\lambda}} \right].$$
The process of extracting this coefficient clearly commutes with the action of $e^{\phi_+(x;t)}$ on states, since $\phi_+(x;t)$ is independent of $\boldsymbol{z}$.
Subsequently applying the commutation relation from Proposition~\ref{commutation}
\begin{align*} \Braket{ \mu | e^{\phi_+(x;t)} \psi_{-1/2} | \lambda } & = (-1)^n \Braket{ \mu | e^{\phi_+(x;t)} \psi^\ast(z_1) \cdots \psi^\ast(z_n) \psi_{-\frac{1}{2}} | \emptyset } \, \Big\vert_{\boldsymbol{z}^{\lambda}} \\
& = (-1)^n \left[ e^{\sum_{q \geq 1} s_q^{(1)}(x_1;t) \sum_{j=1}^n z_j^q} \Braket{ \mu | \psi^\ast(z_1) \cdots \psi^\ast(z_n) \psi_{-\frac{1}{2}} | \emptyset } \right] \, \Big\vert_{\boldsymbol{z}^{\lambda}} 
\end{align*}
where we have used that $e^{\phi_+(x;t)} \psi_{-\frac{1}{2}} \Ket{ \emptyset} = \psi_{-\frac{1}{2}}\Ket{ \emptyset}$ in the second equality, since no particles can move to lower index under a $J_q$ in the shifted vacuum state. The bra-ket above can be explicitly computed. By orthogonality, the creation operators must act exactly at the parts of $\mu$ (shifted by $1/2$) and at $-1/2$, in order to pair with the bra $\Bra{\mu}$. Letting $\mu_- = (\mu_1, \ldots, \mu_{n-1}, 0)$,
$$ \Braket{ \mu | \psi^\ast(z_1) \cdots \psi^\ast(z_n) \psi_{-\frac{1}{2}} |\emptyset } = \sum_{w \in S_n} (-1)^{\ell(w)} \boldsymbol{z}^{w(\mu_-)}, $$
where, as usual, the elements $w$ act on $\mu_-$ by permuting the parts.

Moreover as in Lemma~\ref{sumtheseries},
$$ e^{\sum_{q \geq 1} s_q^{(1)}(x;t) \sum_{j=1}^n z_j^q} = \prod_{1 \leq j \leq n} \frac{(1+t z_j x)}{(1-z_j x)}. $$
So we are left to compute:
\begin{equation} (-1)^n \left[ \sum_{w \in S_n} (-1)^{\ell(w)} \boldsymbol{z}^{w(\mu_-)} \prod_{1 \leq j \leq n} \frac{(1+t z_j x)}{(1-z_j x)} \right]  \, \Big\vert_{\boldsymbol{z}^{\lambda}}. \label{expansiontocompute} \end{equation}
We claim that will be 0 unless $\mu_i \leq \lambda_i$ for all $i \in [1, n-1]$.  To see this, suppose that $\mu_i > \lambda_i$ for some $i$. Let $\omega \in S_n$ be given.  If there exists $j \leq i$ so that $\omega$ takes $j$ to $i$, then the exponent of $z_i$ in $\mathbf{z}^{\omega(\mu_-)}$ is $\mu_j> \mu_i > \lambda_i$.  Note that the expansion of $\prod \frac{(1+tz_jx)}{(1-z_jx)}$ only has non-negative powers of each $z$, and hence the coefficient of $\mathbf{z}^\lambda$ in $\mathbf{z}^{w(\mu_-)}\prod \frac{(1+tz_jx)}{(1-z_jx)}$ is zero.  If, on the other hand, $i$ is the image of some element from $\{i+1,\cdots,n\}$ under $w$, then by the pigeon hole principle there exists some $k>i$ which is the image of some $j \leq i$ under $w$.  Using the same argument as before (but examining the coefficient of $z_k$), we again get a zero coefficient for $\mathbf{z}^\lambda$ in $\mathbf{z}^{w(\mu_-)}\prod \frac{(1+tz_jx)}{(1-z_jx)}$. 


Now suppose that $\mu_i < \lambda_{i+1}$, and let $w \in S_n$ be given.  Define $\tilde w = w s_i$, where $s_i$ is the simple reflection $(i,i+1)$.  Using the identity 
\begin{equation}\label{eq:power.series.identity}\frac{(1+tzx)}{(1-zx)} = 1 + (1+t)\sum_{\ell \geq 1} (zx)^\ell
\end{equation}
and the fact that $\mu_i < \lambda_{i+1}$ implies $\mu_i < \lambda_i$, we find that $$\mathbf{z}^{w(\mu_-)}\prod_{1 \leq j \leq n} \frac{(1+tz_jx)}{(1-z_jx)} \, \Big \vert_{\boldsymbol{z}^\lambda} = \mathbf{z}^{\tilde w(\mu_-)}\prod_{1 \leq j \leq n} \frac{(1+tz_jx)}{(1-z_jx)} \, \Big \vert_{\boldsymbol{z}^\lambda}.$$  Since $\ell(w) +1 = \ell(\tilde w)$, the sum (\ref{expansiontocompute}) is zero.  Hence (\ref{expansiontocompute}) is zero unless $\mu_i \geq \lambda_{i+1}$ for each $i$; combined with the result from the previous paragraph, this means that (\ref{expansiontocompute}) is zero unless $\mu$ interleaves $\lambda$.

So suppose that $\mu$ interleaves $\lambda$, and we calculate (\ref{expansiontocompute}).  
Let $\mathcal{I} = \{1 \leq i \leq n-1: \mu_i = \lambda_{i+1}\}$; note that $|\mathcal{I}| = r(\lambda;\mu)$ by definition.  For each such $i \in \mathcal{I}$, let $s_{i}$ be the simple reflection $(i,i+1)$.  For a nonempty subset $J = \{i_1 < \cdots < i_k\} \subseteq \mathcal{I}$, define $\sigma_J = s_{i_1} \cdots s_{i_k}$; when $J = \emptyset$, define $\sigma_J = e$.  Then the contributions to $\mathbf{z}^\lambda$ come from those elements of $S_n$ of the form $\sigma_J$ for some $J \subseteq \mathcal{I}$, and 
\begin{align*}\left[(-1)^{\ell(\sigma_J)}\mathbf{z}^{\sigma_J(\mu_-)}\prod_{1 \leq j \leq n}\frac{(1+tz_jx)}{(1-z_j x)}\right] \Big\vert_{\mathbf{z}^\lambda} &= \left[(-1)^{\ell(\sigma_J)}\prod_{1 \leq j \leq n} \frac{(1+tz_jx)}{(1-z_jx)}\right]\Big\vert_{\mathbf{z}^{\lambda-\sigma_J(\mu_-)}}\\
&= (-1)^{|J|}(1+t)^{s(\lambda;\mu)+1-|J|}x^{|\lambda|-|\mu|},
\end{align*}
according to the identity (\ref{eq:power.series.identity}).  
Hence (\ref{expansiontocompute}) is
\begin{align*}
(-1)^n\sum_{J \subseteq I} (-1)^{|J|}(1+t)^{s(\lambda;\mu)+1-|J|}x^{|\lambda|-|\mu|}&=(-1)^n(1+t)^{s(\lambda;\mu)+1-|\mathcal{I}|}x^{|\lambda|-|\mu|}\left(\sum_{J \subseteq I} (-1)^{|J|}(1+t)^{|\mathcal{I}|-|J|}\right)\\&=(-1)^nt^{|\mathcal{I}|}(1+t)^{s(\lambda;\mu)+1-|\mathcal{I}|}x^{|\lambda|-|\mu|}.
\end{align*}
\end{proof}

\begin{proof}[Proof of Theorem~\ref{th:factorization.of.partition.function}] The proof starts much as in the previous result:
\begin{align*} \psi_{-1/2} \left| \lambda \right> & = (-1)^n \psi_{\lambda_1 + n - 1/2}^\ast  \psi_{\lambda_2 + n - 3/2}^\ast \cdots \psi_{\lambda_n + 1/2}^\ast \psi_{-1/2} \left| 0 \right> \\
& = (-1)^n \psi^\ast(z_1) \psi^\ast(z_2) \cdots \psi^\ast(z_n) \left| -1 \right> \, \Big\vert_{\boldsymbol{z}^{\lambda}}.
\end{align*}
The process of extracting this coefficient clearly commutes with the action of $e^{\phi_+(x_i;t_i)}$ on states, since $\phi_+(x_i;t_i)$ is independent of $\boldsymbol{z}$ for all $i$. Applying the commutation relation from Proposition~\ref{commutation} to the rightmost $e^{\phi_+(x_n;t_n)}$ in the bra-ket:
\begin{multline} \left< \emptyset \right| \prod_{i=1}^n \left[ e^{\phi_+(x_i;t_i)} \psi_{-1/2} \right] \left| \lambda \right> 
\\ = (-1)^n \left[ e^{\sum_{q \geq 1} s_q^{(1)}(x_n; t_n) \sum_{i=1}^n z_i^q} \left< 0 \right| \prod_{i=1}^{n-1} \left[ e^{\phi_+(x_i;t_i)} \psi_{-1/2} \right] \psi^\ast(z_1) \psi^\ast(z_2) \cdots \psi^\ast(z_n) \left| -1 \right> \right] \, \Big\vert_{\boldsymbol{z}^{\lambda}} \label{firstcommutation} \\
= (-1)^n \left[ \prod_{1 \leq j \leq n} \frac{(1+t_n z_j x_n)}{(1-z_j x_n)} \left< 0 \right| \prod_{i=1}^{n-1} \left[ e^{\phi_+(x_i;t_i)} \psi_{-1/2} \right] \psi^\ast(z_1) \psi^\ast(z_2) \cdots \psi^\ast(z_n) \left| -1 \right> \right] \Big\vert_{\boldsymbol{z}^{\lambda}}
\end{multline}
where we have used that $e^{\phi_+(x_n;t_n)} \left| -1 \right> = \left| -1 \right>$ in the first equality, since no particles can move to lower index under a $J_q$ in this shifted vacuum state. We have used Lemma~\ref{sumtheseries} in the last equality.
The resulting bra-ket will be 0, owing to the rightmost $\psi_{-1/2}$, unless one of the $\psi^\ast(z_j)$ creates a particle at $-1/2$ for some $j \in [1,n]$. This creation operator at $-1/2$ comes with a power $z_j^{-1}$. Thus we may rewrite the bra-ket appearing in (\ref{firstcommutation}), pushing the rightmost $\psi_{-1/2}$ past the creation operators $\psi^\ast(z_j)$, as a sum of $n$ bra-kets as follows:
$$ (-1)^{n-1} \sum_{j=1}^n (-1)^{n-j} z_j^{-1} \left< 0 \right| \prod_{i=1}^{n-2} \left[ e^{\phi_+(x_i;t_i)} \psi_{-1/2} \right] e^{\phi_+(x_{n-1}; t_{n-1})} \psi^\ast(z_1) \psi^\ast(z_2) \cdots \hat{\psi}^\ast(z_j) \cdots \psi^\ast(z_n) \left| -1 \right>,  $$
where the $\hat\psi^\ast(z_j)$ means this operator is to be omitted in the product. Then moving $e^{\phi_+(x_{n-1};t_{n-1})}$ rightward in each summand produces:
\begin{multline} \sum_{j=1}^n (-1)^{j+1} z_j^{-1} e^{\sum_{q \geq 1} s_q^{(1)}(x_{n-1}; t_{n-1}) \sum_{i \ne j} z_i^q} \\  \left< 0 \right| \prod_{i=1}^{n-2} \left[ e^{\phi_+(x_i;t_i)} \psi_{-1/2} \right] \psi^\ast(z_1) \psi^\ast(z_2) \cdots \hat{\psi}^\ast(z_j) \cdots \psi^\ast(z_n) \left| -1 \right>,  \end{multline}
Using Lemma~\ref{sumtheseries} to rewrite the above exponential, then (\ref{firstcommutation}) is equal to:
\begin{multline*} (-1)^n \left[ \prod_{k=1}^n \frac{(1+t_n z_k x_n)}{(1-z_k x_n)} \sum_{j=1}^n (-1)^{n-j+1} z_j^{-1} \prod_{k \ne j} \frac{(1+t_{n-1} z_k x_{n-1})}{(1-z_k x_{n-1})} \right. \\
\left. \left< 0 \right| \prod_{i=1}^{n-2} \left[ e^{\phi_+(x_i;t_i)} \psi_{-1/2} \right] \psi^\ast(z_1) \psi^\ast(z_2) \cdots \hat{\psi}^\ast(z_j) \cdots \psi^\ast(z_n) \left| -1 \right> \right] \Big\vert_{\boldsymbol{z}^{\lambda}}
\end{multline*}
We repeat the same process as above to successively move terms $e^{\phi_+(x_i;t_i)} \psi_{-1/2}$ rightward in the bra-ket. Upon completion, we obtain:
\begin{multline}
 \left[ (-1)^{\frac{n(n+1)}{2}} \sum_{\sigma \in S_n} (-1)^{\ell(\sigma)} (z_1 \cdots z_n)^{-1} \prod_{k=1}^n \frac{(1+t_n z_k x_n)}{(1-z_k x_n)}  \prod_{k \ne j_n} \frac{(1+t_{n-1} z_k x_{n-1})}{(1-z_k x_{n-1})} \cdots \right. \\ \left. \cdots \mkern -18mu \prod_{k \ne j_n, \ldots, j_{2}} \mkern -7mu \frac{(1+t_1 z_k x_1)}{(1-z_k x_1)} \right] \Big\vert_{\boldsymbol{z}^{\lambda}},
\label{lotsofprods} \end{multline}
where the partition $\sigma$ for each summand is dictated by the order in which we pick off creation operators $\psi^\ast(z_{j_i})$ required to create a particle at $-1/2$ for each time step, and we have used the notation $(j_1, \ldots, j_n)$ for $(\sigma(1), \ldots, \sigma(n))$. Thus the identity partition $e \in S_n$ corresponds to using $\psi^\ast(z_n)$ to create a particle at $-1/2$, next $\psi^\ast(z_{n-1})$ and so on. Such an ordering produces no minus signs associated to the creation of fermions at $-1/2$ in the sum over $S_n$, as one is always creating the left-most particle at $-1/2$. Moreover, a single simple reflection changes the sign to $(-1)$, resulting in the $(-1)^{\ell(\sigma)}$ inside the summation. The sign $(-1)^{n(n+1)/2}$ comes from commuting each of the $\psi_{-1/2}$ past the creation operators $\psi^\ast(z_j)$ at each stage, and the number of such creation operators at each stage are $n$, $n-1$, \ldots, $1$.

We will use the determinant identity of Cauchy (see for example Equation (7.121) of Stanley \cite{stanley})
\begin{equation} \det_{1 \leq i, j \leq n} \left\{ (1 - x_i z_j)^{-1} \right\} = \frac{\prod_{i<j} (x_i - x_j) (z_i - z_j)}{\prod_{1 \leq i,j \leq n} (1 - z_i x_j)} 
\label{detident} \end{equation}
to rewrite the above result. In view of this, rewrite (\ref{lotsofprods}) in the form
\begin{multline} (-1)^{n(n+1)/2} \left[ (z_1 \cdots z_n)^{-1} \prod_{i,j} (1- z_i x_j)^{-1} \prod_{k=1}^n (1+t_n z_k x_n) \sum_{\sigma \in S_n} (-1)^{\ell(\sigma)} \right. \\ 
\left. \prod_{k \ne j_n} (1+t_{n-1} z_k x_{n-1})(1-z_{j_n} x_{n-1}) \cdots \prod_{\substack{k \ne j_n, \ldots, j_{2} \\ \text{i.e. } k = j_1}} (1+t_1 z_k x_1)(1-z_{j_n} x_1) \cdots (1-z_{j_{2}} x_1) \right] \, \Big\vert_{\boldsymbol{z}^{\lambda}} \label{commondenom}
\end{multline}
We wish to evaluate the terms appearing in the above summation over $S_n$. For $n \geq 2$, let $P_n$ be the polynomial
$$ P_n(\boldsymbol{x}; \boldsymbol{z}; \boldsymbol{t}) := \sum_{\sigma \in S_n} (-1)^{\ell(\sigma)}  \prod_{\ell=1}^{n-1} \left(\left(\prod_{k =1}^{\ell} (1+t_{\ell} z_{j_k} x_{\ell})\right)\left(\prod_{k=\ell+1}^n1-z_{j_k} x_{\ell}\right)\right)
$$
where here we regard $\boldsymbol{x} = (x_1, \ldots, x_{n-1})$ and $\boldsymbol{t} = (t_1, \ldots, t_{n-1})$ and as usual $\boldsymbol{z} = (z_1, \ldots, z_n)$.

To evaluate, we make a variable change, setting $t_i x_i = y_i$ for $i=1, \ldots, n-1$. Let $\hat{P}_n$ be the corresponding polynomial:
$$ \hat{P}_n(\boldsymbol{x}; \boldsymbol{z}; \boldsymbol{y}) := \sum_{\sigma \in S_n} (-1)^{\ell(\sigma)}  \prod_{k \ne j_n} (1+ y_{n-1} z_k)(1-z_{j_n} x_{n-1}) \cdots \mkern -18mu \prod_{\substack{k \ne j_n, \ldots, j_{2} \\ \text{i.e. } k = j_1}} (1+ y_1 z_k)(1-z_{j_n} x_1) \cdots (1-z_{j_{2}} x_1). $$

\begin{lemma} For any $n \geq 2$, 
$$ \hat{P}_n = \prod_{i=1}^{n-1} (y_i+x_i) \prod_{1 \leq i < j \leq n-1} (x_i+y_j) \prod_{i < j} (z_i - z_j).$$
\label{divislemma} \end{lemma}

Replacing $y_i = x_i t_i$, we obtain the immediate corollary evaluating $P_n$:

\begin{corollary} For any $n \geq 2$, 
$$ P_n(\boldsymbol{x}; \boldsymbol{z}; \boldsymbol{t}) = x_1 \cdots x_{n-1} \prod_{i=1}^{n-1} (1+t_i) \prod_{1 \leq i < j \leq n-1} (x_i+t_j x_j) \prod_{i < j} (z_i - z_j).$$
\end{corollary}

\begin{proof}[Proof of Lemma~\ref{divislemma}]
Given a pair of indices $i, j$, performing the change of variables $\sigma \mapsto (i \, j) \sigma$, where $(i \, j)$ is a transposition, we find that
$$ \hat{P}_n(\boldsymbol{x}; \boldsymbol{z}; \boldsymbol{y}) = - \hat{P}_n(\boldsymbol{x}; \boldsymbol{z}^{(i \, j)}; \boldsymbol{y}) $$
where $\boldsymbol{z}^{(i \, j)}$ denotes the vector $(z_1, \ldots, z_n)$ with roles of $z_i$ and $z_j$ interchanged. Thus $\hat{P}_n$ is antisymmetric in the $z_j$ and so divisible by $\prod_{i < j} z_i - z_j$. We may divide by this factor, and the resulting quotient $\hat{Q}_n$ is a symmetric polynomial with respect to the $z_j$, $j=1, \ldots, n$. We claim that $\hat{Q}_n$ is constant in the $z_j$. Indeed, the terms in $\hat{P}_n$ have degree at most $n-1$ in any single $z_j$, but 
$$ \hat{Q}_n \prod_{i < j} (z_i - z_j) = \hat{Q}_n \sum_{\sigma \in S_n} (-1)^{\ell(\sigma)} \boldsymbol{z}^{\sigma(\rho)}; $$
so $\hat{Q}_n$ must have degree 0 in all $z_j$, else it would contradict this degree bound.

Now the monomials in $\hat{P}_n$ have total degree in both $\boldsymbol{x}$ and $\boldsymbol{y}$ equal to the degree in $\boldsymbol{z}$, and hence {\it all} terms in $\hat{P}_n$ must have total degree $|\rho| = n(n-1)/2$ in $\boldsymbol{x}$ and $\boldsymbol{y}$. This implies that if the quantity appearing on the right-hand side in the statement of the lemma divides $\hat{P}_n$, then it must be equal to $\hat{P}_n$ (up to an easily determined absolute constant, say by setting $x_i=1$ and $y_i=0$ for all $i \in [1, n-1]$). 

To see that $\hat{P}_n$ is divisible by $x_{i} + y_{i}$ for any $i = 1, \ldots, n-1$, note that terms in the product containing $x_i$ and $y_i$ come from a product over $k \not\in \{ j_n, \ldots, j_{i+1} \}$. Setting $x_i = -y_i$ renders all binomials in this product equal, and so the indices $j_{i+1}$ and $j_{i}$ in $\hat{P}_n$ have symmetric roles. Thus using the transposition $t_i = (i+1, i)$ to partition $S_n$ into pairs $\{ \sigma, \sigma t_i \}$, we see that $\hat{P}_n = 0$ when $x_i = -y_i$, since the summands from each pair have length different by one and roles of $j_{i+1}$ and $j_{i}$ interchanged.

To finish the proof, we will show that $\hat{P}_n$ is divisible by $x_i + y_j$ for all pairs $i,j$ with $1 \leq i < j \leq n-1$. The proof is a more complicated version of the argument for divisibility by $y_i + x_i$ above. For $x_i + y_{i+1}$ with $i \leq n-2$, we use the transposition $t_{i, i+1} := (i, i+2)$ to partition $S_n$ into pairs $\{ \sigma, \sigma t_{i, i+1} \}$. These two summands for $\sigma$ and $\sigma t_{i, i+1}$ have a large common factor $C(\boldsymbol{x}; \boldsymbol{z}; \boldsymbol{y})$, and the length of their corresponding permutations differs by one. Remembering that $(j_1, \ldots, j_n) = (\sigma(1), \ldots, \sigma(n))$, their sum may be expressed as $C(\boldsymbol{x}; \boldsymbol{z}; \boldsymbol{y})$ times:
\begin{multline} (1 + y_i z_{j_{i}}) (1 - x_i z_{j_{i+2}})(1+y_{i+1}z_{j_i})(1-x_{i+1}z_{j_{i+2}}) - \\ (1+y_iz_{j_{i+2}})(1-x_iz_{j_i})(1+y_{i+1}z_{j_{i+2}})(1-x_{i+1}z_{j_{i}})  \label{uncommonfactors} \end{multline}
We see that the expression in (\ref{uncommonfactors}) vanishes for all such pairs $\{ \sigma, \sigma t_{i, i+1} \}$ upon identifying the sets $\{ x_i, x_{i+1} \}$ and $\{ -y_{i}, -y_{i+1} \}$. In particular for any $i \in [1, \ldots, n-2]$, setting $x_i = -y_{i+1}$ and $x_{i+1} = -y_i$ shows that $\hat{P}_n$ is in the radical ideal $\left< x_i + y_{i+1}, y_i+x_{i+1} \right>$.

Similarly for any $1 \leq i < j \leq n-1$, using the transposition $t_{i,j} = (i,j+1)$ to partition $S_n$ into pairs $\{ \sigma, \sigma t_{i,j} \}$, we see that $\hat{P}_n$ vanishes upon identifying the sets $\{ x_i, \ldots, x_j \}$ and $\{ -y_i, \ldots, -y_j \}$. Making the identifications $x_i = -y_j$ and $x_{i+1} = -y_i, \ldots, x_{j} = -y_{j-1}$ shows that $\hat{P}_n$ is in the radical ideal $\langle x_i + y_j, x_{i+1} + y_i, \ldots, x_j + y_{j-1} \rangle$. 

A simple degree counting argument shows that $\hat{P}_n$ must in fact be divisible by $x_i + y_{j}$ for $1 \leq i < j \leq n-1.$ Indeed, consider the special case $(i,j)=(n-2,n-1)$. According to the definition of $\hat{P}_n$, $x_{n-1}$ can appear with degree at most 1 in any monomial. Writing
$$ \hat{P}_n = \left[ R_1(\boldsymbol{x}; \boldsymbol{y}) (x_{n-1}+y_{n-2}) + R_2(\boldsymbol{x}; \boldsymbol{y}) (x_{n-2}+y_{n-1}) \right] \prod_{i < j} (z_i - z_j), $$
then we'd like to conclude that $R_1(\boldsymbol{x}; \boldsymbol{y}) = 0$ and so $\hat{P}_n$ is divisible by $x_{n-2} + y_{n-1}$. If we set $y_{n-1} = -x_{n-2}$ and $R_1 \ne 0$ then $R_1$ must be divisible by $x_{n-1} - x_{n-2}$ (since $\hat{P}_n$ was divisible by $x_{n-1} + y_{n-1}$ before specializing). But this contradicts the fact that all monomials in $\hat{P}_n$ have degree at most 1 in $x_{n-1}$. The following inductive argument similarly shows that for all $1 \leq i < j \leq n-1$, we have divisibility by $x_i + y_j$. 

Suppose $\hat{P}_n$ is divisible by $x_{i'} + y_{j'}$ for all pairs $(i',j')$ with $j' \geq i' > i$. Since $\hat{P}_n$ belongs to the radical ideal $\langle x_i + y_j, x_{i+1} + y_i, \ldots, x_j + y_{j-1} \rangle$, we may write
$$ \hat{P}_n = \left[ R_1(\boldsymbol{x}; \boldsymbol{y}) (x_i+y_j) + R_2(\boldsymbol{x}; \boldsymbol{y}) (x_{i+1}+y_i) + \cdots + R_{j+1-i}(\boldsymbol{x}; \boldsymbol{y}) (x_{j}+y_{j-1}) \right] \prod_{i < j} (z_i - z_j). $$
For a given $\ell \in [i+1,\cdots,j]$, set $y_j = -x_i$ and $x_{k} = -y_{k-1}$ for all $k \in [i+1, \ldots, j]$ with $k \ne \ell$. In this specialization, the only remaining term above is 
$$ \bar{R}_{\ell+1-i}(\boldsymbol{x}; \boldsymbol{y}) (x_{\ell}+y_{\ell-1}) \prod_{i < j} (z_i - z_j), $$
where $\bar{R}_{\ell+1-i}$ indicates we have specialized variables in $R_{\ell+1-i}$ as above. But according to our induction hypothesis, $\hat{P}_n$ is divisible by $x_{\ell}+y_k$ for $\ell \leq k <n$. Since all monomials in $\hat{P}_n$ have degree at most $n-\ell$ in $x_{\ell}$ then $R_{\ell+1-i}$ must be 0. Since $\ell$ was arbitrary, all $R_k$ with $k > 1$ vanish and $\hat{P}_n$ is divisible by $x_i + y_j$.
\end{proof}

Applying the corollary to (\ref{commondenom}), then $\left< \emptyset \right| \prod_{i=1}^n \left[ e^{\phi_+(x_i;t_i)} \psi_{-1/2} \right] \left| \lambda \right>$ is equal to
\begin{multline} (-1)^{n(n+1)/2} x_1 \cdots x_{n-1} \prod_{i=1}^{n-1} (1+t_i) \prod_{1 \leq i < j \leq n-1} (x_i + t_jx_j) \\
\left[ (z_1 \cdots z_n)^{-1} \prod_{i,j} (1- z_i x_j)^{-1} \prod_{k=1}^n (1+t_n z_k x_n) \prod_{i < j} (z_i - z_j) \right] \, \Big\vert_{\boldsymbol{z}^{\lambda}}
\end{multline}
and applying the Cauchy identity (\ref{detident}) to the term in brackets results in
\begin{multline} (-1)^{n(n+1)/2} x_1 \cdots x_{n-1} \prod_{i=1}^{n-1} (1+t_i) \prod_{1 \leq i < j \leq n-1} (x_i + t_jx_j) \prod_{i < j} (x_i - x_j)^{-1}  \\
\left[ \prod_{k=1}^n (z_k^{-1}+t_n x_n) \det_{1 \leq i, j \leq n} \left\{(1-x_i z_j)^{-1}\right\}  \right] \, \Big\vert_{\boldsymbol{z}^{\lambda}}.
\end{multline}

Thus to finish the theorem, it remains to show that
\begin{equation} \prod_{i < j} (x_i - x_j)^{-1} \left[ \prod_{k=1}^n (z_k^{-1}+t_n x_n) \det_{1 \leq i, j \leq n} \left\{(1-x_i z_j)^{-1}\right\}  \right] \, \Big\vert_{\boldsymbol{z}^{\lambda}} = 
\prod_{k=1}^n (x_k + t_nx_n) s_{\lambda-\rho}(\boldsymbol{x}). \label{lasttoshow} \end{equation}
The right-hand side of \eqref{lasttoshow} can be expanded as:
$$ \sum_{i=0}^n e_i(\boldsymbol{x}) (t_n x_n)^{n-i} s_{\lambda-\rho}(\boldsymbol{x}) $$
where $e_i$ is the $i$-th elementary symmetric polynomial. Next we apply Pieri's formula, which states:
$$ s_\nu(\boldsymbol{x}) e_i(\boldsymbol{x}) = \sum_{\mu \in \nu \otimes 1^i} s_{\mu} $$
where $\nu \otimes 1^i$ denotes the set of partitions obtained by adding $i$ boxes (at most one per column) to the Ferrers diagram for $\nu$.
Thus we arrive at:
\begin{equation} \sum_{i=0}^n e_i(\boldsymbol{x}) (t_n x_n)^{n-i} s_{\lambda-\rho}(\boldsymbol{x}) = \sum_{i=0}^n (t_n x_n)^{n-i} \sum_{\mu \in (\lambda-\rho) \otimes 1^i} s_{\mu}(\boldsymbol{x}), \label{rightsidemassaged} \end{equation}
where the sum over $\mu$ is again as in Pieri's formula.

But noting that
$$ \det_{1 \leq i, j \leq n} \left\{(1-x_i z_j)^{-1}\right\} \, \Big\vert_{\boldsymbol{z}^{\nu}} = \det_{i,j} (x_i^{\nu_j}) $$
then the left-hand side of \eqref{lasttoshow} is:
$$ \prod_{i < j} (x_i - x_j)^{-1} \left[ \prod_{k=1}^n (z_k^{-1}+t_n x_n) \det_{1 \leq i, j \leq n} \left\{(1-x_i z_j)^{-1}\right\}  \right] \, \Big\vert_{\boldsymbol{z}^{\lambda}} = \sum_{i=0}^n (t_n x_n)^{n-i} \sum_{\nu = \lambda \otimes 1^i} s_{\nu - \rho}(\boldsymbol{x}). $$
As this matches \eqref{rightsidemassaged}, the proof is complete.
\end{proof}

\section{Models for other Cartan types}

Ice models for other classicial groups, whose partition functions result in deformations of highest weight characters, were considered in \cite{bbcg, bs, hamel-king, hk-tokuyama-type, ivanov}. Their admissible states are again rectangular lattices in the six-vertex model whose Boltzmann weights match or closely resemble the weights of type $\Gamma$ and $\Delta$ given above, but one side of one boundary of the ice is modified. The modified boundary identifies pairs of edges using a ``u-turn" which contains a single vertex with its own local weight. An example of such a modified boundary appears in Figure \ref{fig:nonnested.ice.and.pattern} below. The ``u-turn'' bend may identify consecutive rows of ice (a ``non-nested'' bend as in \cite{bbcg, ivanov, hamel-king}) or identify rows symmetric about the middle of the state of ice (``nested'' bends as in \cite{bs, hk-tokuyama-type}), reflecting various embeddings of the classical group into the general linear group. The models with non-nested bends are amenable to study by the Hamiltonian techniques of this paper, and we demonstrate this with various examples in the following section.

For a strictly decreasing partition $\lambda = (\lambda_1>\cdots>\lambda_n)$, a family of ice models is defined in \cite{bbcg} -- which we'll denote $\mathfrak{C}_{\text{NN}}^\lambda$ -- whose underlying graph structure is a lattice of $2n$ rows and $\lambda_1$ columns; columns are numbered as before, but rows are numbered (from top to bottom) by $n, \bar n, n-1, \overline{n-1}, \cdots, 1, \bar 1$.  There is an edge connecting the far right side of rows $j$ and $\bar j$; note that this means that these ``bent edges" are non-nested for this model (hence the ``NN" subscript), in contrast to the numerous families of nested, bent ice models considered in \cite{bs}.  The boundary conditions are as before (though, of course, there is no right boundary for these models).  

As in the rectangular case, one can assign weights to each vertex to arrive at a weight for each state. When the weights are chosen particularly well, the corresponding partition function $\mathcal{Z}(\mathfrak{C}_{\text{NN}}^\lambda)$ has a meaningful Lie-theoretic interpretation.  In \cite{ivanov}, this is achieved with the following scheme: vertices in row $i$ are weighted with $\Delta$-weights from \cite{bbf-ice}, but with each $t_i$ specialized to a common value of $t$; vertices in row $\bar i$ are weighted with $\Gamma'$-weights, where the prime indicates that the index of the parameters is $i$, each $x_i$ is replaced by $x_i^{-1}$, and each $t_i$ specialized to a common value of $t$; an ``upward" bend is weighted $tx_i$; and a ``downward" bend is weighted $x_i^{-1}$.   We call these the $\Delta\Gamma'$ weights. In \cite{bbcg} a similar weighting scheme on $\mathfrak{C}_{\text{NN}}^\lambda$  --- but with up weight $-\sqrt{-tx_i}$ and down weight $\sqrt{x_i}^{-1}$ --- produces a partition function that is conjecturally related of a metaplectic Whittaker function on $\text{Sp}_{2r}$.

States of $\mathfrak{C}_{\text{NN}}^\lambda$ correspond to a family of patterns akin to the Gelfand-Tsetlin patterns mentioned earlier.  Specifically, let $A$ be a state of $\mathfrak{C}_{\text{NN}}^\lambda$, and for each $1 \leq i \leq n$ define a partition $\lambda^{(i)}$ by recording those columns $j$ so that the vertex at $(i,j)$ has its northern edge decorated ``up;" likewise for each $1 \leq i \leq n$ define $\lambda^{(\bar i)}$ by recording those columns $j$ so that the vertex at $(\bar i,j)$ has its northern edge decorated ``up."  These partitions can be arranged into a half-triangular array, and they satisfy a similar interleaving condition as Gelfand-Tsetlin patterns.  For $1 < i \leq n$, the given partitions satisfy one of the following:
\begin{itemize}
\item $\ell(\lambda^{(i)}) = \ell(\lambda^{(\bar i)}) = \ell(\lambda^{(i-1)})+1$, or
\item $\ell(\lambda^{(i)}) = \ell(\lambda^{(\bar i)})+1 = \ell(\lambda^{(i-1)})+1$.
\end{itemize}
(See Figure \ref{fig:nonnested.ice.and.pattern} for an example.)

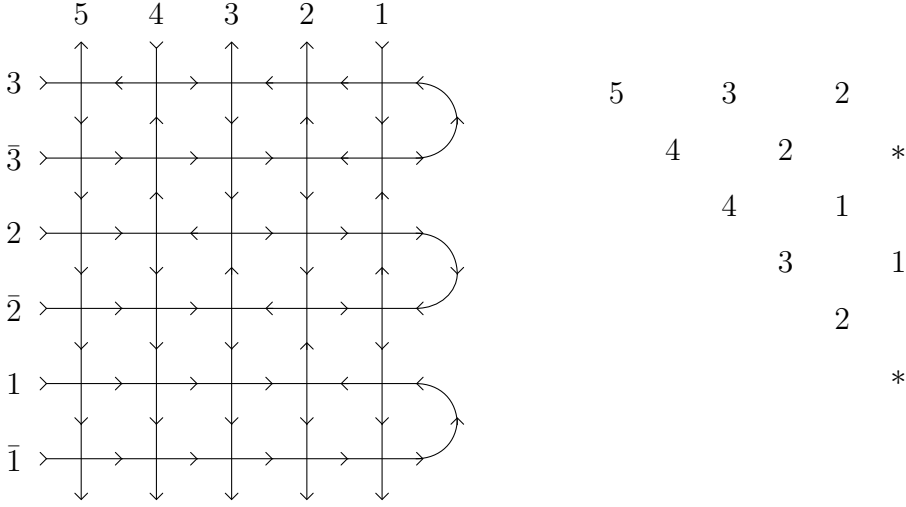
\begin{figure}[!ht]
\begin{subfigure}[h]{.48\textwidth}
  \centering
\begin{tikzpicture}[>=angle 90,shorten <=-1.5pt,shorten >=-1.5pt]

\node [label=left:$3$]      at (0,5) {};
\node [label=left:$\bar 3$] at (0,4) {};
\node [label=left:$2$]      at (0,3) {};
\node [label=left:$\bar 2$] at (0,2) {};
\node [label=left:$1$]      at (0,1) {};
\node [label=left:$\bar 1$] at (0,0) {};

\node [label=above:$5$] at (0.5,5.5) {};
\node [label=above:$4$] at (1.5,5.5) {};
\node [label=above:$3$] at (2.5,5.5) {};
\node [label=above:$2$] at (3.5,5.5) {};
\node [label=above:$1$] at (4.5,5.5) {};

\draw [>-] (0,5) -- (1,5);
\draw [>-] (0,4) -- (1,4);
\draw [>-] (0,3) -- (1,3);
\draw [>-] (0,2) -- (1,2);
\draw [>-] (0,1) -- (1,1);
\draw [>-] (0,0) -- (1,0);

\draw [<-] (0.5,5.5) -- (0.5,4.5);
\draw [>-] (0.5,4.5) -- (0.5,3.5);
\draw [>-] (0.5,3.5) -- (0.5,2.5);
\draw [>-] (0.5,2.5) -- (0.5,1.5);
\draw [>-] (0.5,1.5) -- (0.5,.5);
\draw [>->] (0.5,0.5) -- (0.5,-0.5);
\draw [<-] (1,5) -- (2,5);
\draw [>-] (1,4) -- (2,4);
\draw [>-] (1,3) -- (2,3);
\draw [>-] (1,2) -- (2,2);
\draw [>-] (1,1) -- (2,1);
\draw [>-] (1,0) -- (2,0);

\draw [>-] (1.5,5.5) -- (1.5,4.5);
\draw [<-] (1.5,4.5) -- (1.5,3.5);
\draw [<-] (1.5,3.5) -- (1.5,2.5);
\draw [>-] (1.5,2.5) -- (1.5,1.5);
\draw [>-] (1.5,1.5) -- (1.5,.5);
\draw [>->] (1.5,0.5) -- (1.5,-0.5);
\draw [>-] (2,5) -- (3,5);
\draw [>-] (2,4) -- (3,4);
\draw [<-] (2,3) -- (3,3);
\draw [>-] (2,2) -- (3,2);
\draw [>-] (2,1) -- (3,1);
\draw [>-] (2,0) -- (3,0);

\draw [<-] (2.5,5.5) -- (2.5,4.5);
\draw [>-] (2.5,4.5) -- (2.5,3.5);
\draw [>-] (2.5,3.5) -- (2.5,2.5);
\draw [<-] (2.5,2.5) -- (2.5,1.5);
\draw [>-] (2.5,1.5) -- (2.5,.5);
\draw [>->] (2.5,0.5) -- (2.5,-0.5);
\draw [<-] (3,5) -- (4,5);
\draw [>-] (3,4) -- (4,4);
\draw [>-] (3,3) -- (4,3);
\draw [<-] (3,2) -- (4,2);
\draw [>-] (3,1) -- (4,1);
\draw [>-] (3,0) -- (4,0);

\draw [<-] (3.5,5.5) -- (3.5,4.5);
\draw [<-] (3.5,4.5) -- (3.5,3.5);
\draw [>-] (3.5,3.5) -- (3.5,2.5);
\draw [>-] (3.5,2.5) -- (3.5,1.5);
\draw [<-] (3.5,1.5) -- (3.5,.5);
\draw [>->] (3.5,0.5) -- (3.5,-0.5);
\draw [<-<] (4,5) -- (5,5);
\draw [<->] (4,4) -- (5,4);
\draw [>->] (4,3) -- (5,3);
\draw [>-<] (4,2) -- (5,2);
\draw [<-<] (4,1) -- (5,1);
\draw [>->] (4,0) -- (5,0);

\draw [>-] (4.5,5.5) -- (4.5,4.5);
\draw [>-] (4.5,4.5) -- (4.5,3.5);
\draw [<-] (4.5,3.5) -- (4.5,2.5);
\draw [<-] (4.5,2.5) -- (4.5,1.5);
\draw [>-] (4.5,1.5) -- (4.5,.5);
\draw [>->] (4.5,0.5) -- (4.5,-0.5);
\draw [-]
	(5,5) arc (90:0:.5);
\draw [<-]
	(5.5,4.5) arc (0:-90:.5);

\draw [-]
	(5,3) arc (90:0:.5);
\draw [>-]
	(5.5,2.5) arc (0:-90:.5);

\draw [-]
	(5,1) arc (90:0:.5);
\draw [<-]
	(5.5,.5) arc (0:-90:.5);

\end{tikzpicture}
\end{subfigure} \hspace{20pt}
\begin{subfigure}[h]{.4\textwidth}
\begin{tikzpicture}[scale=.75]
  \centering
\node [label=$5$]      at (1,6) {};
\node [label=$3$]      at (3,6) {};
\node [label=$2$]      at (5,6) {};

\node [label=$4$]      at (2,5) {};
\node [label=$2$]      at (4,5) {};
\node [label=$*$]      at (6,5) {};

\node [label=$4$]      at (3,4) {};
\node [label=$1$]      at (5,4) {};

\node [label=$3$]      at (4,3) {};
\node [label=$1$]      at (6,3) {};

\node [label=$2$]      at (5,2) {};

\node [label=$*$]      at (6,1) {};
\end{tikzpicture}
\end{subfigure}
\caption{An element of $\mathfrak{C}_{\text{NN}}^{(5,3,2)}$ and its corresponding pattern. We've included an additional $*$ for those $\lambda^{(\bar i)}$ with $\ell(\lambda^{(i)})>\ell(\lambda^{(\bar i)})$.}
\label{fig:nonnested.ice.and.pattern}
\end{figure}

Given Corollary~\ref{cor:rectangular.ice.and.fermions}, it is natural to ask whether there is a method for interpreting the $\Delta\Gamma'$-weighting scheme for a state of ice in terms of a discrete time evolution on the associated pattern.  We have the following

\begin{corollary}\label{conj:nonnested.ice.and.fermions}
For $A$ a state of $\mathfrak{C}_{\text{NN}}^\lambda$ and $(i,\bar i)$ a row of $A$, we have $$\text{wt}_{\Delta\Gamma'}(A;(i,\bar i)) = \Braket{\lambda^{(i-1)}|\left(\beta_1 \psi_{-\frac{1}{2}}^* + \beta_2\right)\psi_{\lambda_1+\frac{1}{2}}e^{\phi_-(x_i^{-1};t)}|\lambda^{(\bar i)}}\Braket{\lambda^{(\bar i)}|e^{\phi_+(x_i;t)}\left(\alpha_1 \psi_{-\frac{1}{2}}+\alpha_2\right)|\lambda^{(i)}}$$ where 
\begin{align*}
\alpha_1&= \frac{t}{(-1)^i (t+1)}&&\alpha_2= x_i^{-1}&&
\beta_1 = \frac{(-1)^{i-1}}{(t+1)x_i^{\lambda_1+1}} && \beta_2 = \frac{1}{(t+1)x_i^{\lambda_1}}.
\end{align*}
\end{corollary}

\begin{proof}
Observe first that for a given row $(i,\bar i)$, only one summand from each operator in the bra-kets are relevant; for instance, if $\lambda^{(i)}$ and $\lambda^{(\bar i)}$ have the same number of parts (i.e., if the bend is oriented ``down"), then the bra-ket on the far right is 
\begin{align*}\Braket{\lambda^{(\bar i)}|e^{\phi_+(x_i;t)}\left(\alpha_1 \psi_{-\frac{1}{2}}+\alpha_2\right)|\lambda^{(i)}} &= 
\alpha_1\Braket{\lambda^{(\bar i)}|e^{\phi_+(x_i;t)} \psi_{-\frac{1}{2}}|\lambda^{(i)}} + \alpha_2\Braket{\lambda^{(\bar i)}|e^{\phi_+(x_i;t)}|\lambda^{(i)}}
\\&=\alpha_2\Braket{\lambda^{(\bar i)}|e^{\phi_+(x_i;t)}|\lambda^{(i)}},\end{align*} where the first summand vanishes because $\lambda^{(\bar i)}$ has too many parts to result from a leftward movement of particles of the state $\psi_{-\frac{1}{2}}|\lambda^{(i)}\rangle$.  Similarly for such a row we have 
$$ \Braket{\lambda^{(i-1)}|\left(\beta_1 \psi_{-\frac{1}{2}}^* + \beta_2\right)\psi_{\lambda_1+\frac{1}{2}}e^{\phi_-(x_i^{-1};t)}|\lambda^{(\bar i)}} = \beta_2\Braket{\lambda^{(i-1)}| \psi_{\lambda_1+\frac{1}{2}}e^{\phi_-(x_i^{-1};t)}|\lambda^{(\bar i)}}.$$
Hence we need to verify that the $\Delta \Gamma'$ weight of row $(i,\bar i)$ is
$$\left\{ \begin{array}{ll}
\beta_1\alpha_1\Braket{\lambda^{(i-1)}|\psi_{-\frac{1}{2}}^* \psi_{\lambda_1+\frac{1}{2}}e^{\phi_-(x_i^{-1};t)}|\lambda^{(\bar i)}}\Braket{\lambda^{(\bar i)}|e^{\phi_+(x_i;t)}\psi_{-\frac{1}{2}}|\lambda^{(i)}}&\text{ if bend is up}\\ 
\beta_2\alpha_2\Braket{\lambda^{(i-1)}| \psi_{\lambda_1+\frac{1}{2}}e^{\phi_-(x_i^{-1};t)}|\lambda^{(\bar i)}}\Braket{\lambda^{(\bar i)}|e^{\phi_+(x_i;t)}|\lambda^{(i)}}& \text{ if bend is down}
\end{array}\right.$$
We'll verify the former; the proof of the latter is similar.

Suppose that row $(i,\bar i)$ has its bend oriented up; this means that $\lambda^{(i)}$ has one more part than $\lambda^{(\bar i)}$, and that $\lambda^{(\bar i)}$ and $\lambda^{(i-1)}$ have the same number of parts.  Hence $\lambda^{(\bar i)}$ interleaves $\lambda^{(i)}$, and by Corollary \ref{cor:rectangular.ice.and.fermions} we have the $\Delta\Gamma'$-weight assigned to the upper row $i$ together with the bend is 
$$ \frac{t}{(-1)^i(t+1)} \Braket{\lambda^{(\bar i)}|e^{\phi_+(x_i;t)}\psi_{-\frac{1}{2}}|\lambda^{(i)}}.$$  

In accounting for the $\Delta\Gamma'$-weight associated to the lower row $\bar i$, note that since the bend weight is ``up" the partitions $\lambda^{(\bar i)}$ and $\lambda^{(i-1)}$ have the same number of parts.  Create partitions $\tilde \lambda^{(i-1)}$ and $\tilde \lambda^{(\bar i)}$ by adding one to each part of $\lambda^{(i-1)}$ and $\lambda^{(\bar i)}$, and let $\tilde \lambda_-^{(\bar i)}$ be the partition one gets by inserting an additional part of $1$ to $\tilde \lambda^{(i)}$. Corollary \ref{cor:rectangular.ice.and.fermions} says the $\Delta \Gamma'$-weight attached to the row of ice between $\tilde \lambda^{(\bar i)}_-$ and $\tilde \lambda^{(i-1)}$ is \begin{align*}
\frac{1}{(t+1)x_i^{\lambda_1+1}}\left\langle\tilde \lambda^{(i-1)} \Big \vert \psi_{\lambda_1+\frac{3}{2}}e^{\phi_-(x_i^{-1};t)}\right\vert&\left.\tilde \lambda^{(\bar i)}_-\right\rangle  = 
\frac{(-1)^{i-1}}{(t+1)x_i^{\lambda_1+1}}\Braket{\tilde \lambda^{(i-1)}|\psi_{\frac{1}{2}}^\ast \psi_{\lambda_1+\frac{3}{2}}e^{\phi_-(x_i^{-1};t)}|\tilde \lambda^{(\bar i)}}  \\
&= \frac{(-1)^{i-1}}{(t+1)x_i^{\lambda_1+1}}\Braket{\lambda^{(i-1)};-1|\psi_{-\frac{1}{2}}^\ast\psi_{\lambda_1+\frac{1}{2}}e^{\phi_-(x_i^{-1};t)}|\lambda^{(\bar i)};-1} \\
&= \frac{(-1)^{i-1}}{(t+1)x_i^{\lambda_1+1}}\Braket{\lambda^{(i-1)}|\psi_{-\frac{1}{2}}^\ast\psi_{\lambda_1+\frac{1}{2}}e^{\phi_-(x_i^{-1};t)}|\lambda^{(\bar i)}}. \\\end{align*}
\end{proof}

\begin{remark*}
One gets a similar result that connects the weighting scheme used in \cite{bbcg} to a bra-ket evaluation as in the previous corollary by replacing $\alpha_1$ and $\alpha_2$ above with 
\begin{align*}\tilde \alpha_1&= \frac{-\sqrt{-tx_i}}{(-1)^i x_i(t+1)}&&\tilde \alpha_2= \sqrt{x_i^{-1}}.\end{align*}
\end{remark*}

\section*{Appendix -- Wick's theorem and Jacobi-Trudi type identities}
\renewcommand{\thesection}{A}

We return to examining the equality $$ \langle \mu ; n-1 \mid e^{H_+[s(t)]} \mid \lambda; n \rangle =  \det_{1 \leq p,q \leq n} h_{\lambda_q-\mu_p-q+p}[x \mid y]$$ (i.e., the first equation from (\ref{plainsuperhamtau})) which presents a formula for the $\tau$-function of the superalgebra Hamiltonian of Definition \ref{superhamdef}. 

\subsection{Two ingredients: Time evolution of fermions and Wick's theorem}

Regarding the $t_q$ as discrete time parameters, then the time evolution of fermionic fields is expressed in the following result (whose proof is formally identical to that of Proposition~\ref{commutation}):

\begin{lemma} For any choice of $\boldsymbol{t} = \{ t_q \}$,
\begin{align*}
 e^{H[\boldsymbol{t}]} \psi(z) e^{-H[\boldsymbol{t}]} &= e^{-\sum_{q \geq 1} t_q z^q} \psi(z)\\
e^{H[\boldsymbol{t}]} \psi^\ast(z) e^{-H[\boldsymbol{t}]} &= e^{\sum_{q \geq 1} t_q z^q} \psi^\ast(z).
\end{align*}

\end{lemma}

We use the version of Wick's theorem as in \cite{Zinn-Justin}.
For any infinite set of parameters $\boldsymbol{t} := \{ t_q \}$, define
$$ \psi_k[\boldsymbol{t}] := e^{H[\boldsymbol{t}]} \psi_k e^{-H[\boldsymbol{t}]} \quad \text{and} \quad \psi_k^\ast[\boldsymbol{t}] := e^{H[\boldsymbol{t}]} \psi_k^\ast e^{-H[\boldsymbol{t}]}. $$

\begin{lemma}[Wick's Theorem]
$$ \langle \ell \mid \psi_{i_1}[0] \psi_{i_2}[0] \cdots \psi_{i_n}[0] \psi^\ast_{j_1}[\boldsymbol{t}] \cdots \psi^\ast_{j_n}[\boldsymbol{t}] \mid \ell \rangle = \det_{1 \leq p,q \leq n} \langle \ell \mid \psi_{i_p}[0] \psi^\ast_{j_q}[\boldsymbol{t}] \mid \ell \rangle. $$
\end{lemma}

Wick's theorem may be proved by induction on $n$.

\subsection{Proof of Jacobi-Trudi identities} As explained in \cite{Zinn-Justin}, the above two ingredients lead to a simple proof of the Jacobi-Trudi
identities, whose proof we include here for completeness.

\begin{theorem} Let $\mu, \lambda$ be partitions with exactly $n$ parts (and parts equal to $0$ allowed). Then for any choice of $\ell$,
$$  \langle \mu ; \ell \mid e^{H[\boldsymbol{t}]} \mid \lambda; \ell \rangle = \det_{1 \leq p,q \leq n} h_{\lambda_q-\mu_p-q+p}[\boldsymbol{t}] $$
where
$$ h_k[\boldsymbol{t}] := \left[ e^{\sum_{q \geq 1} t_q z^q} \right] \mid_{z^{k}}, $$
and the notation $\mid_{z^k}$ indicates that we take the coefficient of $z^k$ in the given expression.
\end{theorem}

\begin{proof}
First note that
$$ \langle \mu ; 0 \mid e^{H[\boldsymbol{t}]} \mid \lambda ; 0 \rangle = \langle -n \mid \psi_{\mu_n-n+1/2} \cdots \psi_{\mu_1 - 1/2} e^{H[\boldsymbol{t}]} \psi^\ast_{\lambda_1 - 1/2} \cdots \psi^\ast_{\lambda_n-n+1/2} \mid -n \rangle, $$
where we've padded $\mu$ and/or $\lambda$ with additional 0's so that each has $n$ parts. This equals
\begin{align*}
\langle -n \mid \psi_{\mu_n-n+1/2}[0] \cdots \psi_{\mu_1 - 1/2}[0] \psi^\ast_{\lambda_1 - 1/2}[\boldsymbol{t}] \cdots \psi^\ast_{\lambda_n-n+1/2}[\boldsymbol{t}] e^{H[\boldsymbol{t}]} \mid -n \rangle,
\end{align*}
since $\psi_k[\boldsymbol{t}] = e^{H[\boldsymbol{t}]} \psi_k e^{-H[\boldsymbol{t}]}$ and hence $\psi_{k}[0] = \psi_k$. Remembering that
$$ H[\boldsymbol{t}] = \sum_q J_q t_q, $$
and $J_q$ moves fermions $q$ units to the left, then $e^{H[\boldsymbol{t}]} \mid -n \rangle = \, \mid \! -n \rangle$, as the only surviving term in the series
expansion of $e^{H[\boldsymbol{t}]}$ is the $1$ (since no fermions in $\mid -n \rangle$ may be displaced leftward).  We then apply Wick's theorem to
\begin{align*}
& \langle -n \mid \psi_{\mu_n-n+1/2}[0] \cdots \psi_{\mu_1 - 1/2}[0] \psi^\ast_{\lambda_1 - 1/2}[\boldsymbol{t}] \cdots \psi^\ast_{\lambda_n-n+1/2}[\boldsymbol{t}] \mid -n \rangle \\
& = \det_{1 \leq p,q \leq n} \langle -n \mid \psi_{\mu_p-p+1/2}[0] \psi^\ast_{\lambda_q-q+1/2}[\boldsymbol{t}] \mid -n \rangle 
= \det_{1 \leq p,q \leq n} \langle -n \mid \psi_{\mu_p-p+1/2}[0] \psi^\ast_{\lambda_q-q+1/2}[\boldsymbol{t}] e^{H[\boldsymbol{t}]} \mid -n \rangle \\
& = \det_{1 \leq p,q \leq n} \langle -n \mid \psi_{\mu_p-p+1/2} e^{H[\boldsymbol{t}]} \psi^\ast_{\lambda_q-q+1/2} \mid -n \rangle.
\end{align*}
As $\lambda_q \geq 0$, then $\psi^\ast_{\lambda_q-q+1/2} \mid -n \rangle$ is non-zero, and the above bra-ket clearly only depends on the difference between $\mu_p-p+1/2$ and $\lambda_q-q+1/2$ (and not on the choice of $-n$ in the bra and ket). Thus replacing $-n$ by $0$ in the bra-ket and rewriting we have:
$$ \det_{1 \leq p,q \leq n} \langle 0 \mid \psi_{1/2} e^{H[\boldsymbol{t}]} \psi^\ast_{\lambda_q-\mu_p-q+p+1/2} \mid 0 \rangle = \det_{1 \leq p,q \leq n} \langle 1 \mid e^{H[\boldsymbol{t}]} \psi^\ast_{\lambda_q-\mu_p-q+p+1/2} \mid 0 \rangle. $$
Finally, using that
$$ e^{H[\boldsymbol{t}]} \psi^\ast(z) = \psi^\ast(z) e^{H[\boldsymbol{t}]} e^{\sum_{q \geq 1} t_q z^q}, $$
then
\begin{align*} h_k[\boldsymbol{t}] & := \langle 1 \mid e^{H[\boldsymbol{t}]} \psi^\ast_{k+1/2} \mid 0 \rangle = \langle 1 \mid e^{H[\boldsymbol{t}]} \psi^\ast(z) \mid 0 \rangle \mid_{z^{k}} \\
& = \left[ e^{\sum_{q \geq 1} t_q z^q} \langle 1 \mid \psi^\ast(z) \mid 0 \rangle \right] \mid_{z^{k}} = \left[ e^{\sum_{q \geq 1} t_q z^q} \right] \mid_{z^{k}}.
\end{align*}
This gives the result.
\end{proof}

\subsection{Symmetric functions as tau functions}

\subsubsection{skew Schur functions}

If we make the substitution $t_q = \frac{1}{q} \sum_{i=1}^n x_i^q$, and write the resulting $h_k[\boldsymbol{t}] = h_k[x_1, \ldots, x_n]$, these
are the familiar homogeneous complete symmetric functions in $x_1, \ldots, x_n$ explicitly given as
$$ h_k[x_1, \ldots, x_n] = \sum_{1 \leq i_1 \leq i_2 \leq \cdots \leq i_k \leq n} x_{i_1} \cdots x_{i_k}. $$
Indeed, with this substitution,
$$  e^{\sum_{q \geq 1} t_q z^q} = \prod_{i=1}^n e^{-\log (1 - zx_i)} = \prod_{i=1}^n (1 - zx_i)^{-1}, $$
whose coefficient of $z^k$ is precisely $h_k[x_1, \ldots, x_n]$.
Applying this to the above theorem results in the determinant expression appearing in the familiar Jacobi-Trudi identity:
$$  \langle \mu ; n \mid e^{H[\boldsymbol{t}]} \mid \lambda; n \rangle =  \det_{1 \leq p,q \leq n} h_{\lambda_q-\mu_p-q+p}[x_1, \ldots, x_n]. $$ 
This demonstrates that $\langle \mu ; n \mid e^{H[\boldsymbol{t}]} \mid \lambda; n \rangle$ is the skew Schur function commonly denoted by
$s_{\lambda / \mu}(x_1, \ldots, x_n).$

\subsubsection{supersymmetric skew Schur functions} More generally, we take
$$ t_q = \frac{1}{q} \left[ \sum_{i=1}^n x_i^q - \sum_{j=1}^m (-y_j)^q \right], $$
so that
$$  e^{\sum_{q \geq 1} t_q z^q} = \prod_{i=1}^n (1 - zx_i)^{-1} \prod_{j=1}^m (1 + zy_j). $$
The coefficient of $z^k$ in this expression, explicitly given by
\begin{equation} h_k[x_1, \ldots, x_n \mid y_1, \ldots, y_m] := \sum_{i=0}^k h_i[x_1, \ldots, x_n] e_{k-i}[y_1, \ldots, y_m] \label{splithk} \end{equation}
where $h_i$ is the $i$-th complete homogeneous symmetric polynomial as above and $e_j$ is the $j$-th elementary symmetric polynomial defined by
$$ e_j[y_1, \ldots, y_m] := \sum_{1 \leq i_1 < \cdots < i_j \leq m} y_{i_1} \cdots y_{i_j}.$$

Hence we may express
$$ \langle \mu ; n \mid e^{H[\boldsymbol{t}]} \mid \lambda; n \rangle =  \det_{1 \leq p,q \leq n} h_{\lambda_q-\mu_p-q+p}[x_1, \ldots, x_n \mid y_1, \ldots y_m]. $$
This is the Jacobi-Trudi identity for supersymmetric skew Schur functions. (See for example Section 6 of Macdonald \cite{Mac-themevar}, 
formula before displayed equation (6.21). His notation uses $\boldsymbol{x} \mid \mid \boldsymbol{y}$ as opposed to a single $\mid$. The non-skew version (i.e. $\mu = \emptyset$) is also
nicely presented in Equation (1.7) of Moens and Van der Jeugt \cite{MoensVanderJeugt}.)

\subsubsection{Tokuyama's generating function}

We keep the Hamiltonian $H[\boldsymbol{t}]$ as in the prior example with $n=1$; that is,
$$ t_q := \frac{1}{q} (x^q - (-y)^q) $$
In order to obtain Tokuyama's generating function, consider the modified bra-kets of the form:
$$ \langle \mu ; n-1 \mid e^{H[\boldsymbol{t}]}  \psi_{-1/2} \mid \lambda; n \rangle, $$
where $\mu$ has $n-1$ parts and $\lambda$ has $n$ parts.

Recall that the state $\left| \lambda; n \right>$ is defined by
$$ \left| \lambda; n \right>  = \psi_{\lambda_1 + n - 1/2}^\ast  \psi_{\lambda_2 + n - 3/2}^\ast \cdots \psi_{\lambda_q - q + n + 1/2}^\ast \cdots  \psi_{\lambda_n + 1/2}^\ast \left| 0 \right> $$  
so that
$$ \psi_{-1/2} \left| \lambda ; n \right> = (-1)^n \psi_{\lambda_1 + n - 1/2}^\ast  \psi_{\lambda_2 + n - 3/2}^\ast \cdots \psi_{\lambda_n + 1/2}^\ast \left| -1 \right>.$$
Thus our bra-ket can be rewritten by our earlier definitions as
$$ (-1)^n \langle \mu ; n-1 \mid \psi_{\lambda_1 + n - 1/2}^\ast[\boldsymbol{t}]  \psi_{\lambda_2 + n - 3/2}^\ast[\boldsymbol{t}] \cdots \psi_{\lambda_n + 1/2}^\ast[\boldsymbol{t}] e^{H[\boldsymbol{t}]} \mid -1 \rangle $$
and again $e^{H[\boldsymbol{t}]} \mid -1 \rangle = \mid -1 \rangle$, so $e^{H[\boldsymbol{t}]}$ can be removed from the bra-ket. Now
$$ \langle \mu;n-1 \mid = \langle 0 \mid \psi_{\mu_{n-1}+1/2} \cdots \psi_{\mu_p -p + n-1/2} \cdots \psi_{\mu_1 + n-3/2} $$
If we pad $\mu$ with an additional part $\mu_n = 0$, we may rewrite this as:
$$ \langle -1 \mid \psi_{\mu_n - 1/2} \psi_{\mu_{n-1}+1/2} \cdots \psi_{\mu_p -p + n-1/2} \cdots \psi_{\mu_1 + n-3/2} $$
Applying Wick's theorem to
\begin{align*}
& \langle -1 \mid \psi_{\mu_n - 1/2} \psi_{\mu_{n-1}+1/2} \cdots \psi_{\mu_1 + n-3/2} \psi^\ast_{\lambda_1 +n- 1/2}[\boldsymbol{t}] \cdots \psi^\ast_{\lambda_n+1/2}[\boldsymbol{t}] \mid -1 \rangle \\
& = \det_{1 \leq p,q \leq n} \langle -1 \mid \psi_{\mu_p-p+n-1/2}[0] \psi^\ast_{\lambda_q-q+n+1/2}[\boldsymbol{t}] \mid -1 \rangle \\
& = \det_{1 \leq p,q \leq n} \langle -1 \mid \psi_{\mu_p-p+n-1/2} e^{H[\boldsymbol{t}]} \psi^\ast_{\lambda_q-q+n+1/2} \mid -1 \rangle \\
& =  \det_{1 \leq p,q \leq n} \langle -1 \mid \psi_{-1/2} e^{H[\boldsymbol{t}]} \psi^\ast_{\lambda_q-\mu_p-q+p+1/2} \mid -1 \rangle \\
& =  \det_{1 \leq p,q \leq n} \langle 0 \mid e^{H[\boldsymbol{t}]} \psi^\ast_{\lambda_q-\mu_p-q+p+1/2} \mid -1 \rangle.
\end{align*}

Finally, using that
$$ e^{H[\boldsymbol{t}]} \psi^\ast(z) = \psi^\ast(z) e^{H[\boldsymbol{t}]} e^{\sum_{q \geq 1} t_q z^q}, $$
we have
\begin{align*} &  \langle 0 \mid e^{H[\boldsymbol{t}]} \psi^\ast_{k+1/2} \mid -1 \rangle = \langle 0 \mid e^{H[\boldsymbol{t}]} \psi^\ast(z) \mid -1 \rangle \mid_{z^{k}} \\
& = \left[ e^{\sum_{q \geq 1} t_q z^q} \langle 0 \mid \psi^\ast(z) \mid -1 \rangle \right] \mid_{z^{k}} = \left[ e^{\sum_{q \geq 1} t_q z^q} \right] \mid_{z^{k+1}} =: h_{k+1}[\boldsymbol{t}].
\end{align*}
So we conclude that
$$ \langle \mu ; n-1 \mid e^{H[\boldsymbol{t}]}  \psi_{-1/2} \mid \lambda; n \rangle = (-1)^n \det_{1 \leq p,q \leq n} h_{\lambda_q-\mu_p-q+p+1}[\boldsymbol{t}] $$
and this latter expression is $(-1)^n s_{\lambda / \mu}(x \mid y)$ with $t_q$ chosen as above. To summarize:
\begin{proposition} The bra-ket appearing in Theorem~\ref{th:computing.weight.of.one.step}, with 
$$ s_q := \frac{1}{q} (x^q - (-y)^q) $$
is a supersymmetric Schur function:
$$ \langle \mu ; n-1 \mid e^{H[\boldsymbol{s}]}  \psi_{-1/2} \mid \lambda; n \rangle = (-1)^n s_{\lambda / \mu}(x \mid y). $$
In particular, taking $y = tx$, then $s_q = s_q^+(t)$ with $n=1$ as in Definition~\ref{superhamdef}.
\end{proposition}

\bibliographystyle{acm}
\bibliography{vertham}

\end{document}